\documentclass[reqno,twoside,11pt]{amsart}
\usepackage{amsmath, amsfonts, amssymb, esint, amsthm, nicefrac, bbm, dsfont}
\usepackage{epsfig, graphicx}

\newcommand{\dx}{~\mathrm{d}x}
\newcommand{\sym}{\mathrm{sym}}

\newcommand{\Id}{\mathrm{Id}}

\newcommand{\R}{\mathbb{R}}
\def\endproof{\hspace*{\fill}\mbox{\ \rule{.1in}{.1in}}\medskip }

\numberwithin{equation}{section}
\theoremstyle{plain}

\newtheorem*{theorem*}{Theorem}
\newtheorem{theorem}{Theorem}[section]
\newtheorem{lemma}[theorem]{Lemma}
\newtheorem{corollary}[theorem]{Corollary}

\theoremstyle{definition}

\begin{document}
\title [The Monge-Amp\`ere system in dimension two]
{The Monge-Amp\`ere system in dimension two: a further regularity improvement}
\author{Marta Lewicka}
\address{M.L.: University of Pittsburgh, Department of Mathematics, 
139 University Place, Pittsburgh, PA 15260}
\email{lewicka@pitt.edu} 

\date{\today}
\thanks{M.L. was partially supported by NSF grant DMS-2006439. 
}

\begin{abstract}
We prove a convex integration result for the
Monge-Amp\`ere system introduced in \cite{lew_conv},
 in case of dimension $d=2$ and arbitrary codimension $k\geq 1$. 
Our prior result \cite{lew_improved} stated flexibility 
up to the H\"older regularity $\mathcal{C}^{1,\frac{1}{1+
 4/k}}$, whereas presently we achieve flexibility up to 
$\mathcal{C}^{1,1}$ when $k\geq 4$ and up to 
$\mathcal{C}^{1,\frac{2^k-1}{2^{k+1}-1}}$ for any $k$. This first result
uses the approach of K\"allen \cite{Kallen}, while the second result
iterates on the approach of Cao-Hirsch-Inauen \cite{CHI} and agrees with it for $k=1$
at the H\"older regularity up to $\mathcal{C}^{1,1/3}$.
\end{abstract}

\maketitle

\section{Introduction}

In this paper we present a new bound for the admissible H\"older
continuity exponent for weak solutions of the
Monge-Amp\`ere system in dimension $d=2$ and arbitrary codimension $k\geq 1$:
\begin{equation}\label{MA}
\begin{split}
& \mathfrak{Det}\,\nabla^2 v= f \quad \mbox{ in }\;\omega\subset\R^2,\\
& \mbox{where }\; \mathfrak{Det}\,\nabla^2 v =
\langle \partial_{11}v, \partial_{22}v\rangle -
\big|\partial_{12} v\big|^2\quad \mbox{ for }\; v:\omega\to\R^k.
\end{split}
\end{equation}
The closely related problem of isometric immersions of a given Riemannian metric $g$:
\begin{equation}\label{II}
\begin{split}
& (\nabla u)^T\nabla u = g \quad\mbox{ in }\;\omega,
\\ &  \mbox{for }\; u:\omega\to \R^{2+k},
\end{split}
\end{equation} 
reduces to (\ref{MA}) upon
taking the family of metrics $\{g_\epsilon=\Id_2+\epsilon A\}_{\epsilon\to 0}$,
each a small perturbation of $\Id_2$, making an ansatz $u_\epsilon =
\mbox{id}_2+ \epsilon v +\epsilon^2 w$, and gathering the lowest
order terms in the $\epsilon$-expansions. This leads to the following system:
\begin{equation}\label{VK}
\begin{split}
& \frac{1}{2}(\nabla v)^T\nabla v + \sym \nabla w = A\quad\mbox{ in }\;\omega,\\
& \mbox{for }\; v:\omega\to \R^k, \quad w:\omega\to\R^2.
\end{split}
\end{equation}
On a simply connected $\omega$, the system (\ref{VK}) is 
equivalent to: $\mbox{curl}\,\mbox{curl}\, \big(\frac{1}{2}(\nabla
v)^T\nabla v \big) = \mbox{curl}\,\mbox{curl}\, A$, and further to:
$\mathfrak{Det}\,\nabla^2 v= -\mbox{curl}\,\mbox{curl}\, A$,
reflecting the agreement of the Gaussian curvatures of $g_\epsilon$
and the surfaces $u_\epsilon (\omega)$ at their lowest order terms in $\epsilon$, and 
bringing us back to (\ref{MA}).

\medskip

\noindent Systems (\ref{MA}) and (\ref{VK}) were introduced and studied in \cite{lew_conv} for
arbitrary dimensions $d$ and $k$, where their flexibility, in the sense of
Theorem \ref{th_final} below, was proved up to the regularity
$\mathcal{C}^{1,\frac{1}{1+d(d+1)/k}}$.
For $d=2$, this means flexibility up
to $\mathcal{C}^{1,\frac{1}{1+6/k}}$, and when $k=1$ this result agrees
with flexibility up to $\mathcal{C}^{1,\frac{1}{7}}$ obtained in
\cite{lewpak_MA}, which was subsequently improved to
$\mathcal{C}^{1,\frac{1}{5}}$ in \cite{CS} (and to $\mathcal{C}^{1,\frac{1}{1+4/k}}$ for $k$ arbitrary in
\cite{lew_improved}), and further to $\mathcal{C}^{1,\frac{1}{3}}$ in \cite{CHI}.
The main purpose of the present paper is to increase the
aforementioned H\"older exponents in the case of codimension $k>1$, obtaining flexibility
up to $\mathcal{C}^{1,1}$ when $k\geq 4$ and up to
$\mathcal{C}^{1,\frac{2^k-1}{2^{k+1}-1}}$ for arbitrary $k\geq
1$. Consequently, this sets the state of the art in
establishing the density of $\mathcal{C}^{1,\alpha}$ solutions to
(\ref{MA}) in $\mathcal{C}^0$, as follows:
$$k=1 \leadsto  \alpha<1/3, \quad k=2 \leadsto
\alpha<3/7, \quad k=3 \leadsto \alpha<7/15, \quad k\geq 4 \leadsto \alpha<1.$$
The large gap between exponents at $k=3$
and $k=4$ is due to the two different techniques in the
Nash-Kuiper iteration scheme: for $k\geq 4$ we use the approach of
K\"allen \cite{Kallen}, while for arbitrary $k$ we iterate the
approach of Cao-Hirsch-Inauen designed in \cite{CHI} at $k=1$.
This approach seemingly yields only the exponent $1/2$ as
$k\to\infty$. Combining the two approaches towards a better
interpolation between their corresponding exponents is an open
problem. 

\smallskip

\noindent We also point out that according to a result due to Poznak (see
\cite[Chapter 2.3]{HH}), any smooth $2$-dimensional metric has a
smooth local embedding in $\R^4$, namely a solution of (\ref{II})
with $k=2$. Our type of density results, albait only addressing the H\"older
continuous solutions, are stronger in the following sense: rather than
yielding existence of a single solution, they
imply that an arbitrary subsolution to (\ref{VK}) or (\ref{II}) can be
approximated by a $\mathcal{C}^{1,\alpha}$ solution.

\bigskip

\noindent Indeed, our main result pertaining to (\ref{VK}) states that a $\mathcal{C}^1$-regular pair $(v,w)$ 
which is a subsolution, can be uniformly approximated
by exact solutions $\{(v_n,w_n)\}_{n=1}^\infty$, as follows:

\begin{theorem}\label{th_final}
Let $\omega\subset\R^2$ be an open, bounded domain and let
$k\geq 1$ be the gi\-ven co\-di\-men\-sion. Given the fields $v\in\mathcal{C}^1(\bar\omega,\R^k)$,
$w\in\mathcal{C}^1(\bar\omega,\R^2)$ and 
$A\in\mathcal{C}^{0,\beta}(\bar\omega,\R^{2\times 2}_\sym)$, assume that:
$$\mathcal{D}=A-\big(\frac{1}{2}(\nabla v)^T\nabla v + \sym\nabla
w\big) \quad\mbox{ satisfies } \quad \mathcal{D}>c\,\Id_d \; \mbox{ on } \; \bar\omega,$$
for some $c>0$, in the sense of matrix inequalities. Then, for every
exponent $\alpha$ with:
\begin{equation}\label{VKrange} 
\begin{array}{ll} 
\displaystyle{\alpha<\min\Big\{\frac{\beta}{2},1\Big\} }& \mbox{for } k\geq 4\vspace{1mm}\\
\displaystyle{\alpha<\min\Big\{\frac{\beta}{2},\frac{2^k-1}{2^{k+1}-1}}\Big\} & \mbox{for any } k,
\end{array}
\end{equation}
and for every $\epsilon>0$, 
there exists $\tilde v\in\mathcal{C}^{1,\alpha}(\bar \omega,\R^k)$,
$\tilde w\in\mathcal{C}^{1,\alpha}(\bar\omega,\R^d)$ such that the following holds:
\begin{align*}
& \|\tilde v - v\|_0\leq \epsilon, \quad \|\tilde w - w\|_0\leq \epsilon,
\tag*{(\theequation)$_1$}\refstepcounter{equation} \label{FFbound1}\vspace{1mm}\\ 
& A -\big(\frac{1}{2}(\nabla \tilde v)^T\nabla \tilde v + \sym\nabla
\tilde w\big) =0 \quad \mbox{ in }\;\bar\omega.
\tag*{(\theequation)$_2$} \label{FFbound2}
\end{align*}
\end{theorem}

\noindent The above result implies the aforementioned density of
solutions to (\ref{MA}), as in \cite{lew_conv}:

\begin{corollary}\label{th_CI_weakMA} 
For any $f\in L^{\infty} (\omega, \R)$ on an open, bounded, simply connected
domain $\omega\subset\mathbb{R}^2$, the following holds.
Fix $k\geq 1$ and fix an exponent $\alpha$ in the range (\ref{VKrange}).
Then the set of $\mathcal{\mathcal{C}}^{1,\alpha}(\bar\omega, \R^k)$
weak solutions to (\ref{MA}) is dense in $\mathcal{\mathcal{C}}^0(\bar\omega, \R^k)$. 
Namely, every $v\in \mathcal{\mathcal{C}}^0(\bar\omega,\R^k)$ is the
uniform limit of some sequence $\{v_n\in\mathcal{\mathcal{C}}^{1,\alpha}(\bar\omega,\R^k)\}_{n=1}^\infty$,
such that:
\begin{equation*} 
\mathfrak{Det}\, \nabla^2 v_n  = f \quad \mbox{ on } \; \omega\;
\mbox{ for all }\; n=1\ldots\infty.
\end{equation*}
\end{corollary} 

\medskip

\noindent The main new technical ingredient allowing for the flexibility
stated in Theorem \ref{th_final}, is the ``stage''-type
constructions in the following two results:

\begin{theorem}\label{thm_stage}
Let $\omega\subset\R^2$ be an open, bounded, smooth planar domain and
let $k\geq 4$. Fix an exponent $\gamma\in (0,1)$ and an integer $N\geq 1$. Then,
there exists $l_0\in (0,1)$ depending only on $\omega$, and there exists
$\sigma_0\geq 1$ depending on $\omega,\gamma, N$,
such that the following holds. Given $v\in\mathcal{C}^2(\bar\omega+\bar B_{2l}(0),\R^k)$, 
$w\in\mathcal{C}^2(\bar\omega+\bar B_{2l}(0),\R^2)$, 
$A\in\mathcal{C}^{0,\beta}(\bar\omega+\bar B_{2l}(0),\R^{2\times 2}_\sym)$ defined on
the closed $2l$-neighbourhood of $\omega$, and given
constants $l, \lambda, M>0$ with the properties:
\begin{equation}\label{Assu}
l\leq l_0,\qquad \lambda^{1-\gamma} l\geq\sigma_0, \qquad
M\geq\max\{\|v\|_2, \|w\|_2, 1\},
\end{equation}
there exist $\tilde v\in\mathcal{C}^2(\bar \omega+\bar B_{l}(0),\R^k)$,
$\tilde w\in\mathcal{C}^2(\bar\omega+\bar B_{l}(0),\R^2)$ such that, denoting the defects:
$$\mathcal{D}=A -\big(\frac{1}{2}(\nabla v)^T\nabla v + \sym\nabla
w\big), \qquad \tilde{\mathcal{D}} =A -\big(\frac{1}{2}(\nabla \tilde
v)^T\nabla \tilde v + \sym\nabla\tilde w\big), $$
the following bounds are valid:
\begin{align*}
& \begin{array}{l}
\|\tilde v - v\|_1\leq C\lambda^{\gamma/2}\big(\|\mathcal{D}\|_0^{1/2} + lM\big), \vspace{1.5mm}\\
\|\tilde w -w\|_1\leq C\lambda^{\gamma}\big(\|\mathcal{D}\|_0^{1/2}
+ lM\big) \big(1+ \|\mathcal{D}\|_0^{1/2} +lM +\|\nabla v\|_0\big), 
\end{array}\vspace{3mm} \tag*{(\theequation)$_1$}\refstepcounter{equation} \label{Abound1}\\
&  \begin{array}{l}
\|\nabla^2\tilde v\|_0\leq C \lambda
\lambda^{\gamma/2}\big(\|\mathcal{D}\|_0^{1/2} + lM\big), \vspace{1.5mm}\\
\|\nabla^2\tilde w\|_0\leq C \lambda
\lambda^{\gamma}\big(\|\mathcal{D}\|_0^{1/2} + lM\big)  \big(1+\|\mathcal{D}\|_0^{1/2} + lM
+ \|\nabla v\|_0\big), 
\end{array} \vspace{3mm} \tag*{(\theequation)$_2$}\label{Abound2} \\ 
& \begin{array}{l}
\|\tilde{\mathcal{D}}\|_0\leq C\big(l^\beta \|A\|_{0,\beta} +
\displaystyle{\frac{\lambda^\gamma}{(\lambda l)^N}}\big(
\|\mathcal{D}\|_0 +(lM)^2\big). 
\end{array} \tag*{(\theequation)$_3$} \label{Abound3}
\end{align*}
The norms of the maps $v, w, A, \mathcal{D}$ and $\tilde v,
\tilde w, \tilde{\mathcal{D}}$ are taken on the respective domains of
the maps' definiteness.
The constants $C$ depend only on $\omega, \gamma, N$.
\end{theorem}

\smallskip


\begin{theorem}\label{thm_stageCHI}
Let $\omega\subset\R^2$ be an open, bounded, smooth planar
domain. Given any codimension $k\geq 1$, the result in Theorem
\ref{thm_stage} holds true, for each sufficiently small $\gamma$ (in
function of $N$, $k$), with \ref{Abound1} -- \ref{Abound3} replaced with:
\begin{align*}
& \begin{array}{l}
\|\tilde v - v\|_1\leq C\lambda^{\gamma/2}\big(\|\mathcal{D}\|_0^{1/2} + lM\big), \vspace{1.5mm}\\
\|\tilde w -w\|_1\leq C\lambda^{\gamma}\big(\|\mathcal{D}\|_0^{1/2}
+ lM\big) \big(1+ \|\mathcal{D}\|_0^{1/2} +lM +\|\nabla v\|_0\big), 
\end{array}\vspace{3mm} \tag*{(\theequation)$_1$}\refstepcounter{equation} \label{Abound12}\\
&  \begin{array}{l}
\|\nabla^2\tilde v\|_0\leq C \displaystyle{\frac{(\lambda
  l)^{N+1}\lambda^{\gamma/2}}{l}}\big(\|\mathcal{D}\|_0^{1/2} + lM\big), 
\vspace{1.5mm}\\ 
\|\nabla^2\tilde w\|_0\leq C \displaystyle{\frac{(\lambda l)^{N+1}\lambda^{\gamma}}{l}}
\big(\|\mathcal{D}\|_0^{1/2} + lM\big)  \big(1+\|\mathcal{D}\|_0^{1/2} + lM+ \|\nabla v\|_0\big), 
\end{array} \vspace{3mm} \tag*{(\theequation)$_2$}\label{Abound22} \\ 
& \begin{array}{l}
\|\tilde{\mathcal{D}}\|_0\leq C\Big(l^\beta \|A\|_{0,\beta} +
\displaystyle{\frac{\lambda^{\gamma}}{(\lambda
    l)^{\frac{2(N^2-1)}{N+3}\big(1-\big(\frac{N-1}{2(N+1)}\big)^k\big)}}}\big(
\|\mathcal{D}\|_0 +(lM)^2\big)\Big). 
\end{array} \tag*{(\theequation)$_3$} \label{Abound32}
\end{align*}
\end{theorem}

\noindent 
By assigning $N$ sufficiently large, we see that the quotient $r$ of the
blow-up rate of $\|\nabla^2\tilde v\|_0$
with respect to the rate of decay of $\|\tilde{\mathcal{D}}\|_0$ can be taken
arbitrarily close to $0$ in Theorem \ref{thm_stage} and arbitrarily
close to $\frac{1}{2\big(1-\frac{1}{2^k}\big)}$ in Theorem \ref{thm_stageCHI}.
Since the H\"older regularity exponent equals $\frac{1}{1+2r}$ (see
section \ref{sec4} and Theorem \ref{th_old}), this
implies the respective ranges in (\ref{VKrange}).

\medskip

\noindent The layout of the paper is as follows. In section
\ref{sec_step} we gather the preparatory results: the
mollification and the commutator estimates, followed by two
decompositions of the $\R^{2\times 2}_\sym$-valued matrix
fields on $\omega\subset\R^2$, into a symmetric gradient and a multiple
of $\Id_2$. The first decomposition, used towards Theorem
\ref{thm_stage}, passes through solving the Poisson
equation with Dirichlet boundary data and thus it does not commute with
differentiation. The second decomposition, on which Theorem
\ref{thm_stageCHI} relies, uses convolution with the
Poisson kernel and thus it has the desired properties
including commuting with differentiation. We then present two convex
integration ``step'' constructions: the first one with Nash's spirals
as the oscillatory perturbations added to the fields $(v,w)$,
and the second one with Kuiper's corrugations.

\smallskip

\noindent 
Section \ref{sec_stage} contains a proof of Theorem \ref{thm_stage},
where K\"allen's iteration procedure is used in codimension $k\geq
4$ allowing for an almost complete absorption of the diagonalized deficit field $\mathcal{D}$ 
in a single step. 
Section \ref{sec_stage2} contains the proof of
Theorem \ref{thm_stageCHI}, which combines iterating on the
codimension with K\"allen's iterations within each step, that
instead of absorbing the full error as explained before, essentially transfers one of
its (diagonal) modes onto the other existing mode. In \cite{CHI}
this observation allowed to argue that a single convex integration
step, necessitating only one codimension $k=1$, suffices to absorb the first order deficit, thus
yielding the regularity exponent $1/3$. Presently, we 
iterate this construction over $k$ codimensions, superposing
perturbation with appropriately increasing frequencies. Each iteration
then reduces the deficit by a factor commensurate with the ratio of the
current and the previous perturbation frequencies,
while the second derivative of $v$ increases by the factor of the next
frequency times the square root of the deficit, see the estimates \ref{INDU1} - \ref{INDU3}. 
This leads to the relative estimates in \ref{Abound22} and
\ref{Abound32} whereas the $C^1$ norm of the accumulated perturbation
remains controlled as in \ref{Abound12}. An interpolation to
$\mathcal{C}^{1,\alpha}$ and the Nash-Kuiper iteration on stages is
the content of section \ref{sec4}, where we complete the proof of
Theorem \ref{th_final}, frequently referring to \cite{lew_improved}.

\smallskip

\noindent 
We point out that our presentation is modular, namely we separate the
``stage'' estimates in Theorems \ref{thm_stage} and \ref{thm_stageCHI},
from the Nash-Kuiper iteration in Theorem \ref{th_old}. This latter construction
automatically yields a density result with the maximal H\"older regularity
exponent given in function of
the respective rates of second derivative blow-up and of the deficit
decrease in a stage. This way, a future improvement of the stage estimates will
directly yield the regularity improvement as well. This
important point was more convoluted in the presentation of \cite{CHI}.

\smallskip

\subsection{Notation.}
By $\mathbb{R}^{2\times 2}_{\sym}$ we denote the space of symmetric
$2\times 2$ matrices. The space of H\"older continuous vector fields
$\mathcal{C}^{m,\alpha}(\bar\omega,\R^k)$ consists of restrictions of
all $f\in \mathcal{C}^{m,\alpha}(\mathbb{R}^2,\R^k)$ to the closure of
an open, bounded domain $\omega\subset\R^2$. The
$\mathcal{C}^m(\bar\omega,\R^k)$ norm of such restriction is
denoted by $\|f\|_m$, while its H\"older norm in $\mathcal{C}^{m,
  \gamma}(\bar\omega,\R^k)$ is $\|f\|_{m,\gamma}$. 
By $C$ we denote a universal constant which may change from line to
line, but it depends only on the specified parameters.

\section{Preparatory statements}\label{sec_step}

In this section, we gather the regularization, decomposition and
perturbation statements that will be used in the course of the convex
integration constructions. The first lemma below consists of the basic convolution estimates and
the commutator estimate from \cite{CDS}:

\begin{lemma}\label{lem_stima}
Let $\phi\in\mathcal{C}_c^\infty(\R^d,\mathbb{R})$ be a standard
mollifier that is nonnegative, radially symmetric, supported on the
unit ball $B(0,1)\subset\R^d$ and such that $\int_{\mathbb{R}^d} \phi \dx = 1$. Denote: 
$$\phi_l (x) = \frac{1}{l^d}\phi(\frac{x}{l})\quad\mbox{ for all
}\; l\in (0,1], \;  x\in\R^d.$$
Then, for every $f,g\in\mathcal{C}^0(\mathbb{R}^d,\R)$ and every
$m,n\geq 0$ and $\beta\in (0,1]$ there holds:
\begin{align*}
& \|\nabla^{(m)}(f\ast\phi_l)\|_{0} \leq
\frac{C}{l^m}\|f\|_0,\tag*{(\theequation)$_1$}\vspace{1mm} \refstepcounter{equation} \label{stima1}\\
& \|f - f\ast\phi_l\|_0\leq C \min\big\{l^2\|\nabla^{2}f\|_0,
l\|\nabla f\|_0, {l^\beta}\|f\|_{0,\beta}\big\},\tag*{(\theequation)$_2$} \vspace{1mm} \label{stima2}\\
& \|\nabla^{(m)}\big((fg)\ast\phi_l - (f\ast\phi_l)
(g\ast\phi_l)\big)\|_0\leq {C}{l^{2- m}}\|\nabla f\|_{0}
\|\nabla g\|_{0}, \tag*{(\theequation)$_3$} \label{stima4}
\end{align*}
with a constant $C>0$ depending only on the differentiability exponent $m$.
\end{lemma}

\medskip

\noindent The next two auxiliary results are specific to dimension
$d=2$. They allow for the decomposition of the given defect into
a multiple of $\Id_2$ (thus two primitive defects of rank $1$) and a
symmetric gradient, in agreement with the local conformal invariance of
any Riemann metric:

\begin{lemma}\label{lem_diagonal}
Let $\omega\subset\R^2$ be an open, bounded and Lipschitz set. There exist maps: 
$$\bar\Psi: L^2(\omega,\R^{2\times 2}_\sym)\to H^{1}(\omega,\R^2), 
\qquad \bar a: L^2(\omega,\R^{2\times  2}_\sym)\to L^{2}(\omega,\R), $$ 
which are linear, continuous, and such that:
\begin{itemize}
\item[(i)] for all $D\in L^2(\omega,\R^{2\times 2}_\sym)$ there holds:
  $D= \bar a(D)\Id_2 + \sym\nabla \big(\bar\Psi(D)\big),$\vspace{1mm}
\item[(ii)] $\bar\Psi(\Id_2) \equiv 0$ and $\bar a(\Id_2) \equiv 1$ in
  $\omega$, \vspace{1mm}
\item[(iii)] for all $m\geq 0$ and $\gamma\in (0,1)$, if $\omega$ is
  $\mathcal{C}^{m+2,\gamma}$ regular then the maps $\bar\Psi$ and $\bar a$ are continuous from
  $\mathcal{C}^{m,\gamma}(\bar\omega,\R^{2\times 2}_\sym)$ to
  $\mathcal{C}^{m+1,\gamma}(\bar\omega, \R^2)$ and
  to $\mathcal{C}^{m,\gamma}(\bar\omega, \R)$, respectively, so that:
\begin{equation}\label{diag_bounds}
\|\bar\Psi (D)\|_{m+1,\gamma}\leq C \|D\|_{m,\gamma} \mbox{ and }
~ \|\bar a (D)\|_{m,\gamma}\leq C \|D\|_{m,\gamma} \quad \mbox{ for all
}\; D\in L^2(\omega,\R^{2\times 2}_\sym).
\end{equation}
\end{itemize} 
The constants $C$ above depend on $\omega$, $m, \gamma$ but not on
$D$. Also, there exists $l_0>0$ depending only on $\omega$, such that
(\ref{diag_bounds}) are uniform on 
the closed $l$-neighbourhoods $\{\bar\omega+ \bar B_l(0)\}_{l\in (0,l_0)}$ of $\omega$.
\end{lemma}

\smallskip

\noindent The proof of the above \cite[Proposition 3.1]{CS} is direct,
namely one may assign:
$\bar a (D) = D_{11} -\partial_{11} \Delta^{-1}(D_{11}-D_{12})
-\partial_{12}\Delta^{-1}(2D_{12})$, where $\Delta^{-1}$ corresponds to solving the
Poisson problem on $\omega$ with the Dirichlet boundary
condition. Consequently, $\bar a$ does not commute with taking partial derivatives of the
input matrix field $D$. An improvement of this construction, which is
crucial for the proof in \cite{CHI}, is given in:

\begin{lemma}\label{lem_diagonal2}
Given a radius $R>0$ and an exponent $\gamma\in (0,1)$, define the linear
space $E$ consisting of $\mathcal{C}^{0,\gamma}$-regular, $\R^{2\times
  2}_\sym$-valued matrix fields $D$ on the ball $\bar B_R\subset \R^2$,
whose traceless part $\dot D = D-\frac{1}{2}(\mbox{trace}\; D) \Id_2$
is compactly supported in $B_R$. There exist linear maps $\bar\Psi$, $\bar a$ in:
\begin{equation*}
\begin{split}
& E = \big\{D\in \mathcal{C}^{0,\gamma}(\bar B_R,\R^{2\times 2}_\sym); ~ \dot{D}\in
\mathcal{C}^{0,\gamma}_c(B_R,\R^{2\times 2}_\sym)\big\}\\
& \bar\Psi: E \to \mathcal{C}^{1,\gamma}(\bar B_R,\R^2), \qquad \bar a: E\to \mathcal{C}^{0,\gamma}(\bar B_R),
\end{split}
\end{equation*} 
with the following properties:
\begin{itemize}
\item[(i)] for all $D\in E$ there holds: $D=\bar a(D)\Id_2 + \sym\nabla \big(\bar\Psi(D)\big),$
\vspace{1mm}
\item[(ii)] $\bar\Psi(\Id_2) \equiv 0$ and $\bar a(\Id_2) \equiv 1$ in $B_R$, \vspace{1mm}
\item[(iii)] $ \|\bar\Psi (D)\|_{1,\gamma}\leq C \|\dot
  D\|_{0,\gamma}$ and $\|\bar a (D)\|_{0,\gamma}\leq C
  \|D\|_{0,\gamma}$ with constants $C$ depending on $R, \gamma$, 
\vspace{1mm}
\item[(iv)] for all $m\geq 1$, if $D\in E\cap \mathcal{C}^{m,\gamma}(\bar B_R,\R^{2\times 2}_\sym)$
  then $\bar\Psi(D)\in \mathcal{C}^{m+1,\gamma}(\bar B_R,\R^2)$ and 
$\bar a(D)\in \mathcal{C}^{m,\gamma}(\bar B_R)$, and we have:
\begin{equation*}
\partial_I\bar\Psi(D)=\bar\Psi(\partial_ID) , \quad \partial_I\bar a(D)=\bar a(\partial_ID)
\qquad \mbox{ for all } \;  |I|\leq m.
\end{equation*}
\item[(v)]  for all $m\geq 1$, if $D\in E\cap
  \mathcal{C}^{m,\gamma}(\bar B_R,\R^{2\times 2}_\sym)$ and if
  additionally $D_{22}=0$ in $\bar B_R$, then:
$$\|\partial_2^{(s)}\bar a(D)\|_{0,\gamma}\leq C \|\partial_2^{(s)}D\|_{0,\gamma}, \quad
\|\partial_1^{(t+1)}\partial_2^{(s)}\bar a(D)\|_{0,\gamma}
\leq C \|\partial_1^{(t)}\partial_2^{(s+1)}D\|_{0,\gamma},$$
for all $s,t\geq 0$ such that $s\leq m$, $t+s+1\leq m$, and with
$C$ depending only on $R,\gamma$.
\end{itemize} 
\end{lemma}
\begin{proof}
For $f\in \mathcal{C}_c(\R^2)$ we set $\psi[f]= \Gamma\ast
f\in \mathcal{C}^1(\R^2)$ with $\Gamma(x)=\frac{1}{2\pi} \log |x|$, namely:
$$\psi[f](x) =\frac{1}{2\pi} \int_{\R^2}\log |x-y| f(y)\;\mbox{d}y.$$
Recall \cite{GT} that the operator $\psi[\cdot]$ is well defined, linear and
satisfies for any $\gamma\in (0,1)$ and $m\geq 1$:
\begin{align*}
& \,\mbox{ if } f\in\mathcal{C}_c^{0,\gamma}(\R^2) \mbox{ then }
\psi[f]\in\mathcal{C}^{2,\gamma}(\R^2) \mbox{ and } \Delta\psi[f] = f
\mbox{ in } \R^2, \tag*{(\theequation)$_1$}\refstepcounter{equation} \label{pr1a}\vspace{1mm}\\
& \begin{array}{l} \mbox{if } f\in\mathcal{C}_c^{m,\gamma}(\R^2) \mbox{ then } 
\psi[f]\in\mathcal{C}^{m+2,\gamma}(\R^2) \vspace{0.5mm}\\ \qquad\qquad\qquad
\quad \mbox{and } \partial_I \psi[f] = \psi[\partial_If] \mbox{  for any multiindex } I  \mbox{
  with } |I|\leq m.\end{array}
\tag*{(\theequation)$_2$} \label{pr2a}
\end{align*}
Moreover, if $\mbox{supp}\;f\subset B_R$ then the following
holds with a constant $C$ depending only of $R, \gamma$:
\begin{align*}
\|\nabla\psi[f]\|_{\mathcal{C}^{1,\gamma}(\bar B_R)}\leq C \|f\|_{0,\gamma}
\tag*{(\theequation)$_3$} \label{pr3a}
\end{align*}

\noindent Given $D: B_R\to \R^{2\times 2}_\sym$ such that $\dot{D}\in
\mathcal{C}_c(B_R,\R^{2\times 2}_\sym)$, define the continuous vector field on $\R^2$:
$$\bar\Psi(D) =\bar\Psi(\dot D) =
\Big(\partial_1\psi[D_{11}-D_{22}]+\partial_2\psi[2D_{12}], ~\partial_1\psi[2D_{12}]
-\partial_2\psi[D_{11} - D_{22}]\Big).$$
If $\dot{D}\in \mathcal{C}_c^{0,\gamma}(B_R,\R^{2\times 2}_\sym)$ then \ref{pr1a} implies 
that $\bar\Psi(D)\in\mathcal{C}^{1,\gamma}(\R^2,\R^2)$, together with
$\partial_1\bar\Psi^1- \partial_2\bar\Psi^2=\Delta\psi[D_{11}-D_{22}]= D_{11}-D_{22}$
and $\partial_1\bar\Psi^2
+\partial_2\bar\Psi^1=\Delta\psi[2D_{12}]=2D_{12}$, so that:
$$D- \sym\nabla (\bar\Psi(D)) = \big(D_{11}- \partial_1\bar\Psi^1(D)\big)\Id_2 
= \big(D_{22}- \partial_2\bar\Psi^2(D)\big)\Id_2.$$
Consequently, there follows (i) and (ii) if we set, for all $D\in E$:
$$\bar a(D)= D_{11}- \partial_1\bar\Psi^1(D)=D_{22}- \partial_2\bar\Psi^2(D).$$
Clearly, (iii) follows from \ref{pr3a} while (iv) is a consequence of \ref{pr2a}.
For (v), the above definition of $\bar a$ and the assumption $D_{22}=0$, yield:
\begin{equation*}
\begin{split}
& \partial_2^{(s)}\bar a(D) = -\partial_2^{(s+1)}\bar\Psi^2(D) =
-\partial_{12}\psi[2\partial_2^{(s)}D_{12}] + \partial_{22}\psi[\partial_2^{(s)}(D_{11}-D_{22})],\\
& \partial_1^{(t+1)}\partial_2^{(s)}\bar a(D) =
-\partial_1^{(t+1)}\partial_2^{(s+1)}\bar\Psi^2(D) 
\\ & \qquad\qquad\qquad \; =
-\partial_{11}\psi[2 \partial_1^{(t)}\partial_2^{(s+1)}D_{12}] 
+ \partial_{12}\psi[\partial_1^{(t)}\partial_2^{(s+1)}(D_{11}-D_{22})].
\end{split}
\end{equation*}
This ends the proof in virtue of \ref{pr3a}.
\end{proof}

\bigskip

\noindent As the final preparatory result, we recall two different single
``step'' constructions from \cite{lew_conv}. The first one uses Nash's
spirals, necessitating two codimension directions in order to cancel each of
the non-zero entries of the given nonnegative defect in the diagonal form:

\begin{lemma}\label{lem_step}
Assume that $k\geq 4$. Let $v\in \mathcal{C}^2(\R^2, \R^{k})$, $w\in
\mathcal{C}^1(\R^2, \R^{2})$, $\lambda>0$ and $a\in \mathcal{C}^2(\R^2)$. 
Denote $\Gamma(t) = \sin t$, $\bar\Gamma(t) = \cos t$ and define:
\begin{equation}\label{defi_per}
\begin{split}
&\tilde v = v+ \frac{a(x)}{\lambda}
\Big(\Gamma(\lambda x_1)e_1 +\bar\Gamma(\lambda x_1)e_2
+ \Gamma(\lambda x_2)e_3 +\bar\Gamma(\lambda x_3)e_4\Big),  \\
& \tilde w = w -\frac{a(x)}{\lambda} 
\Big(\Gamma(\lambda x_1)\nabla v^1 + \bar\Gamma(\lambda x_1)\nabla v^2
+\Gamma(\lambda x_2)\nabla v^3 + \bar\Gamma(\lambda x_2)\nabla v^4\Big).
\end{split}
\end{equation}
Then, the following identity is valid on $\R^2$:
\begin{equation}\label{step_err}
\begin{split}
& \big(\frac{1}{2}(\nabla \tilde v)^T \nabla \tilde v + \sym\nabla \tilde w\big) - 
\big(\frac{1}{2}(\nabla v)^T \nabla v + \sym\nabla w\big) -
\frac{a(x)^2}{2}\mathrm{Id}_2
\\ & = -\frac{a}{\lambda} \Big(\Gamma(\lambda x_1)\nabla^2 v^1 + \bar\Gamma(\lambda x_1)\nabla^2 v^2
+\Gamma(\lambda x_2)\nabla^2 v^3 + \bar\Gamma(\lambda x_2)\nabla^2 v^4\Big) +
\frac{1}{\lambda^2} \nabla a\otimes \nabla a.
\end{split}
\end{equation}
\end{lemma}

\medskip

\noindent The second ``step'' construction uses Kuiper's corrugations,
in which a single codimension is used to cancel one rank-one defect of the form
$a(x)^2e_i\otimes e_i$:

\begin{lemma}\label{lem_step2}
Let $v\in \mathcal{C}^2(\R^2, \R^{k})$, $w\in \mathcal{C}^1(\R^2,
\R^{2})$, $\lambda>0$ and $a\in \mathcal{C}^2(\R^2)$ be given. Denote:
$\Gamma(t) = 2\sin t$, $\bar\Gamma(t) = -\frac{1}{2}\sin (2t),$ and
for a fixed  $i=1\ldots 2$ and $j=1\ldots k$ define:
\begin{equation}\label{defi_per2}
\tilde v = v + \frac{a(x)}{\lambda} \Gamma(\lambda x_i) e_j,\qquad 
\tilde w = w -\frac{a(x)}{\lambda} \Gamma(\lambda x_i)\nabla v^j
+ \frac{a(x)^2}{\lambda} \bar\Gamma(\lambda x_i)e_i.
\end{equation}
Then, the following identity is valid on $\R^2$:
\begin{equation}\label{step_err2}
\begin{split}
& \big(\frac{1}{2}(\nabla \tilde v)^T \nabla \tilde v + \sym\nabla \tilde w\big) - 
\big(\frac{1}{2}(\nabla v)^T \nabla v + \sym\nabla w\big) - a(x)^2e_i\otimes e_i
\\ & = -\frac{a}{\lambda} \Gamma(\lambda x_i)\nabla^2 v^j 
+ \frac{1}{2\lambda^2}\Gamma(\lambda x_i)^2\nabla a\otimes\nabla a
- \frac{1}{\lambda} \bar\Gamma(\lambda x_i)\, \sym\big(\nabla (a^2)\otimes e_i\big).
\end{split}
\end{equation}
\end{lemma}

\section{A proof of Theorem \ref{thm_stage}}\label{sec_stage}

In the proof below, the constants $C>1$ may change from line to line,
but they depend only on  $\omega$, $\gamma$,
$m$ and $r$ (and thus, ultimately, on $N$), unless specified otherwise.

\bigskip

\noindent {\bf Proof of Theorem \ref{thm_stage}}

\smallskip

{\bf 1. (Preparing the data)} Let $l_0$ be as in Lemma
\ref{lem_diagonal}, and $\phi_l$ as in Lemma \ref{lem_stima}. 
For $l\in (0,l_0]$ we define the following smoothed data functions on
the $l$-thickened set $\bar\omega+\bar B_l(0)$:
$$v_0=v\ast \phi_l,\quad w_0=w\ast \phi_l, \quad A_0=A\ast \phi_l,
\quad {\mathcal{D}}_0= A_0 - \big(\frac{1}{2}(\nabla v_0)^T\nabla v_0 + \sym\nabla w_0\big).$$
From Lemma \ref{lem_stima}, we deduce the initial bounds,
where constants $C$ depend only on $m$ and $\omega$:
\begin{align*}
& \|v_0-v\|_1 + \|w_0-w\|_1 \leq C lM,
\tag*{(\theequation)$_1$}\refstepcounter{equation} \label{pr_stima1}\\
& \|A_0-A\|_0 \leq Cl^\beta\|A\|_{0,\beta}, \tag*{(\theequation)$_2$} \label{pr_stima2}\\
& \|\nabla^{(m+1)}v_0\|_0 + \|\nabla^{(m+1)}w_0\|_0\leq
\frac{C}{l^m} lM\quad \mbox{ for all }\; m\geq 1, \tag*{(\theequation)$_3$} \label{pr_stima3}\\
& \|\nabla^{(m)} \mathcal{D}_0\|_0\leq
\frac{C}{l^m} \big(\|\mathcal{D}\|_0 + (lM)^2\big)\quad \mbox{ for
  all }\; m\geq 0. \tag*{(\theequation)$_4$}\label{pr_stima4} 
\end{align*}
Indeed,  \ref{pr_stima1}, \ref{pr_stima2} follow from \ref{stima2} and
in view of the lower bound
on $M$. Similarly, \ref{pr_stima3} follows by applying \ref{stima1} to
$\nabla^2v$ and $\nabla^2w$ with the differentiability exponent $m-1$.
Since:
$$\mathcal{D}_0 = \mathcal{D}\ast \phi_l - \frac{1}{2}\big((\nabla
v_0)^T\nabla v_0 - ((\nabla v)^T\nabla v)\ast\phi_l\big), $$ 
we get \ref{pr_stima4} by applying \ref{stima1} to $\mathcal{D}$, and
\ref{stima4} to $\nabla v$.

\bigskip

{\bf 2. (Induction definition: iterative decomposition of deficits)} 
Let the linear maps ${\bar a, \bar\Psi}$ be as in 
Lemma \ref{lem_diagonal}. Also, let $r_0>0$ be a constant depending on
$\omega$, $\gamma$, such that $\|\bar a(D)-1\|_0\leq \frac{1}{2}$
whenever $\|D-\mathrm{Id}_2\|_{0,\gamma}\leq r_0$. For $r=0\ldots N$ we iteratively define the 
perturbation amplitudes $a_r \in\mathcal{C}^\infty(\bar\omega+\bar
B_l(0),\R)$ and the correction fields $\Psi_r \in\mathcal{C}^\infty(\bar\omega+\bar
B_l(0),\R^2)$ by setting: 
\begin{equation*}
\begin{split}
& a_0=0, \qquad a_r= \Big(2\bar a\big(\tilde C\Id_2 + \mathcal{D}_0-\mathcal{E}_{r-1}\big)\Big)^{1/2},
\quad \Psi_r = \bar\Psi\big(\tilde C\Id_2 + \mathcal{D}_0-\mathcal{E}_{r-1}\big),
\\ &  \mbox{with } ~\tilde C = \frac{2}{r_0}\Big(\|\mathcal{D}_0\|_{0,\gamma} +
\frac{1}{l^\gamma}\big(\|\mathcal{D}\|_{0} +(lM)^2\big)\Big),
\end{split}
\end{equation*}
and with the error fields $\mathcal{E}_r \in\mathcal{C}^\infty(\bar\omega,\R^{2\times 2}_\sym)$ given
by the right hand side of (\ref{step_err}), namely:
\begin{equation*}
\begin{split}
& \mathcal{E}_{r}= - \frac{a_r}{\lambda} \Big(\Gamma(\lambda
x_1)\nabla^2 v^1 + \bar\Gamma(\lambda x_1)\nabla^2 v^2 +
\Gamma(\lambda x_2)\nabla^2 v^3 + \bar\Gamma(\lambda x_2)\nabla^2 v^4\Big) +
\frac{1}{\lambda^2} \nabla a_r\otimes \nabla a_r.
\end{split}
\end{equation*}
Our definition of $a_r$ is correctly posed if only $\|\mathcal{D}_0
-\mathcal{E}_{r-1}\|_{0,\gamma} \leq r_0\tilde C$.
To this end, we right away observe that $\|\mathcal{D}_0\|_{0,\gamma} \leq
r_0\tilde C/2$, while we will prove that the second condition in (\ref{Assu}) implies:
\begin{equation}\label{need_errK}
\|\mathcal{E}_r\|_{0,\gamma}\leq \frac{r_0\tilde C }{2} \quad \mbox{ for all } \; r=0\ldots N-1.
\end{equation}
Note that then automatically there holds for all $r=1\ldots N$:
\begin{equation}\label{low_bd_asK}
\begin{split}
& \tilde C\Id_2 + \mathcal{D}_0-\mathcal{E}_{r-1} =
\frac{1}{2}(a_r)^2\mathrm{Id}_2 +\sym\nabla\Psi_r
\\ & \mbox{and }\quad (a_r)^2\in [\tilde C, 3\tilde C]
\mbox{ in }\bar\omega +\bar B_l(0).
\end{split}
\end{equation}
For the future estimate of derivatives of $a_r$ of order $m\geq 1$, we
use Fa\'a di Bruno's formula, \ref{pr_stima4} and the bound
(\ref{diag_bounds}) in Lemma \ref{lem_diagonal} in:
\begin{equation}\label{asm}
\begin{split}
\|\nabla^{(m)}a_r\|_0 & \leq C \Big\|\sum_{p_1+2p_2+\ldots
  mp_m=m} a_r^{2(1/2-p_1-\ldots -p_m)}\prod_{t=1}^m \big|\nabla^{(t)}(a_r)^2\big|^{p_t}\Big\|_0\\
& \leq C \sum_{p_1+2p_2+\ldots mp_m=m}\frac{1}{\tilde C^{(p_1+\ldots+p_m)-1/2}}
\prod_{t=1}^m \big( \tilde C + \|\mathcal{D}_0\|_{t,\gamma} +
\|\mathcal{E}_{r-1}\|_{t,\gamma} \big)^{p_t} \\ & \leq {C}{\tilde
  C^{1/2}} \sum_{p_1+2p_2+\ldots mp_m=m}\prod_{t=1}^m
\Big(\frac{1}{l^t} +\frac{\|\mathcal{E}_{r-1}\|_{t,\gamma}}{\tilde C}\Big)^{p_t},
\end{split}
\end{equation}
Additionally, applying Fa\'a di Bruno's formula to the inverse rather
than square root, we get:
\begin{equation}\label{asm2}
\begin{split}
& \|\nabla^{(m)}\Big(\frac{1}{a_{r}+a_{r-1}}\Big)\|_0 \leq \frac{C}{\tilde
C^{1/2}} \sum_{p_1+2p_2+\ldots mp_m=m}\prod_{t=1}^m
\Big(\frac{\|\nabla^{(t)}(a^{r} + a^{r-1})\|_0}{\tilde
  C^{1/2}}\Big)^{p_t} 
\end{split}
\end{equation}
The formulas (\ref{asm}), (\ref{asm2}) hold for all
$r=1\ldots N$ with constants $C$ depending on $\omega$, $\gamma$, $m$.

\bigskip

{\bf 3. (Inductive estimates)} In steps 4-5 we will prove the
following estimates, valid for all $r=1\ldots N$,
with constants $C$ depending, in line with our convention,
only on $\omega$, $\gamma$, $m$, $r$:
\begin{align*}
&  \|a_r\|_0\leq C\tilde C^{1/2} \tag*{(\theequation)$_1$}\refstepcounter{equation} \label{Fbound1K}\medskip\\
& \|\nabla^{(m)} a_r\|_0\leq C
\frac{\lambda^{m}\lambda^{\gamma}}{\lambda l}\tilde C^{1/2}  \quad\mbox{ for all }\; m\geq 1,
\tag*{(\theequation)$_2$}\label{Fbound2K} \medskip\\
& \|\nabla^{(m)} \mathcal{E}_r\|_0\leq C
\frac{\lambda^{m}}{\lambda l}\tilde C
\; \quad \qquad\mbox{ for all }\; m\geq 0, \tag*{(\theequation)$_3$}\label{Fbound22K} \medskip\\
& \|\nabla^{(m)} \big(\mathcal{E}_r-\mathcal{E}_{r-1}\big)\|_0 \leq
C \frac{\lambda^{m}\lambda^{2(r-1)\gamma}}{(\lambda l)^r}\tilde C
\quad\mbox{ for all }\; m\geq 0. \tag*{(\theequation)$_4$}\label{Fbound3K}
\end{align*}
In general, $C\to\infty$ as $m\to \infty$ or $r\to \infty$, so it is crucial that eventually
only finitely many of bounds above are used.
We now check that \ref{Fbound1K}, \ref{Fbound2K} are already valid at 
$r=1$. Indeed, \ref{Fbound1K} is a consequence of 
(\ref{low_bd_asK}) since $\mathcal{E}_0=0$, whereas \ref{Fbound2K} follows from (\ref{asm}):
\begin{equation*}
\|\nabla^{(m)}a_1\|_0  \leq C \frac{\tilde C^{1/2}}{l^m} \leq
C \frac{\lambda^m}{\lambda l} \tilde C^{1/2}.
\end{equation*}
We now observe that \ref{Fbound22K} always follows from 
\ref{Fbound1K} and \ref{Fbound2K}, because:
\begin{equation*}
\begin{split}
& \|\nabla^{(m)}\mathcal{E}_r\|_0\leq
C\sum_{p+q+t=m}\lambda^{p-1}\|\nabla^{(q)} a_r\|_0\|\nabla^{(t+2)}v_0\|_0 + 
C\sum_{q+t=m}\lambda^{-2}\|\nabla^{(q+1)} a_r\|_0\|\nabla^{(t+1)}a_1\|_0 
\\ & \leq C\Big(\sum_{p+t=m}\frac{\lambda^{p-1}}{l^{t+1}}\tilde C^{1/2}(lM)
+ \sum_{p+q+t=m, q\neq 0} \frac{\lambda^{p+q-1}\lambda^\gamma}{(\lambda l) l^{t+1}}\tilde C^{1/2} (lM)
+ \sum_{q+t=m} \frac{\lambda^{q+t+2}\lambda^{2\gamma}}{\lambda^2(\lambda l)^2}\tilde C\Big)
\\ & \leq  C \lambda^{m}\Big(\sum_{p+t=m}\frac{1}{(\lambda l)^{t+1}}
+ \sum_{p+q+t=m, q\neq 0} \frac{\lambda^\gamma}{(\lambda l)^{t+2}}
+ \sum_{q+t=m} \frac{\lambda^{2\gamma}}{(\lambda l)^2}\Big) \tilde C 
\leq  C \frac{\lambda^{m}}{\lambda l}\tilde C,
\end{split}
\end{equation*}
if only ${\lambda^{2\gamma}}\leq {\lambda l}$. 
Strengthening this working assumption to:
\begin{equation}\label{mu}
\frac{\lambda l}{\lambda^{2\gamma}}\geq \frac{2C}{r_0},
\end{equation}
we additionally arrive at (\ref{need_errK}). Concluding and since \ref{Fbound22K},
\ref{Fbound3K} are equivalent at $r=1$, we note that we have
proven \ref{Fbound1K} - \ref{Fbound3K} and (\ref{need_errK})
at their lowest counter $r$ value.

\bigskip

{\bf 4. (Proof of the inductive estimates)} Assume
that the bounds \ref{Fbound1K}--\ref{Fbound3K} hold up to some 
$1\leq r\leq N-1$. We will prove their validity at
$r+1$. By (\ref{need_errK}) and (\ref{low_bd_asK}) we directly get
\ref{Fbound1K}, whereas (\ref{asm}) and \ref{Fbound22K} yield 
\ref{Fbound2K}, since for all $m\geq 1$:
\begin{equation*}
\begin{split}
& \|\nabla^{(m)}a_{r+1}\|_0  \leq C \tilde C^{1/2} \sum_{p_1+2p_2+\ldots mp_m=m}\prod_{t=1}^m
\Big(\frac{1}{l^t} + \frac{\lambda^t\lambda^\gamma}{\lambda l}\Big)^{p_t}
\\ & \leq C \tilde C^{1/2} \sum_{p_1+2p_2+\ldots mp_m=m}\prod_{t=1}^m
\Big(\frac{\lambda^t\lambda^\gamma}{\lambda l}\Big)^{p_t}
\\ & = C \tilde C^{1/2} \lambda^m \sum_{p_1+2p_2+\ldots mp_m=m}
\Big(\frac{\lambda^\gamma}{\lambda l}\Big)^{p_1+p_2+\ldots p_m}\leq 
C\frac{\lambda^{m}\lambda^\gamma}{\lambda l}\tilde C^{1/2}.
\end{split}
\end{equation*}
We have already justified \ref{Fbound22K} and
(\ref{need_errK}) in the previous step, so it remains to show
\ref{Fbound3K}. Towards this end, note first the rough bound below, in view of  (\ref{asm2}) and
\ref{Fbound1K}--\ref{Fbound3K}:
\begin{equation*}
\begin{split}
& \|\nabla^{(m)}\Big(\frac{1}{a_{r+1}+a_{r}}\Big)\|_0 \leq \frac{C}{\tilde
C^{1/2}} \sum_{p_1+2p_2+\ldots mp_m=m}\prod_{t=1}^m
\Big(\frac{\lambda^t\lambda^\gamma}{\lambda l}\Big)^{p_t} \leq \frac{C}{\tilde C^{1/2}} 
\lambda^m.
\end{split}
\end{equation*}
Further, since $(a_{r+1})^2-(a_r)^2 = -2\bar
a(\mathcal{E}_r-\mathcal{E}_{r-1})$, it follows that:
\begin{equation}\label{mu2}
\begin{split}
& \|\nabla^{(m)}\big(a_{r+1}-a_{r}\big)\|_0\leq C \sum_{q+t=m}
\|\nabla^{(q)}\big((a_{r+1})^2-(a_{r})^2\big)\|_0 
\|\nabla^{(t)}\Big(\frac{1}{a^{r+1}+a^{r}}\Big)\|_0 \\ & 
\leq \frac{C}{\tilde C^{1/2}} \sum_{q+t=m} \|\mathcal{E}_{r}-\mathcal{E}_{r-1} \|_{q,\gamma} \lambda^t 
\leq C \tilde C^{1/2} \lambda^m  \frac{\lambda^{(2r-1)\gamma}}{(\lambda l)^r}, 
\end{split}
\end{equation}
for all $m\geq 0$. Finally, writing:
\begin{equation*}
\begin{split}
\mathcal{E}_{r+1}-\mathcal{E}_r = & -\frac{a_{r+1}-a_{r}}{\lambda}
\Big(\Gamma(\lambda x_1)\nabla^2 v_0^1 + 
\bar\Gamma(\lambda x_1)\nabla^2 v_0^2 + \Gamma(\lambda x_2)\nabla^2 v_0^3 + 
\bar\Gamma(\lambda x_2)\nabla^2 v_0^4\Big)  \\ & +
\frac{1}{\lambda^2} \Big(\big(\nabla a_{r+1} - \nabla a_r\big)\otimes \nabla a_{r+1} + 
\nabla a_r \otimes \big(\nabla a_{r+1} - \nabla a_r\big)\Big),
\end{split}
\end{equation*}
we conclude, in virtue of (\ref{mu2}), \ref{Fbound2K} and \ref{pr_stima2}:
\begin{equation*}
\begin{split}
& \|\nabla^{(m)}(\mathcal{E}_{r+1}-\mathcal{E}_r)\|_0 \leq  C 
\sum_{p+q+t=m}\lambda^{p-1}
\|\nabla^{(q)}\big(a_{r+1}-a_{r}\big)\|_0 \|\nabla^{(t+2)}v_0\|_0
\\ & \qquad\qquad\qquad\qquad \quad
+ C \sum_{q+t=m} \lambda^{-2} \|\nabla^{(q+1)} (a_{r+1} - a_r)\|_0
\big(\|\nabla^{(t+1)} a_{r+1} \|_0+ \|\nabla^{(t+1)} a_{r} \|_0\big)
\\ & \leq C \tilde C^{1/2} \sum_{p+q+t=m}\lambda^{p+q-1}
\frac{\lambda^{(2r-1)\gamma}}{(\lambda l)^{r}} \frac{lM}{l^{t+1}}
+ C\tilde C\sum_{q+t=m} \lambda^{-2} \frac{\lambda^{q+1}\lambda^{(2r-1)\gamma}}{(\lambda l)^{r}}
\frac{\lambda^{t+1}\lambda^\gamma}{\lambda l} 
\\ & \leq C\tilde C\lambda^m \sum_{p+q+t=m}
\frac{\lambda^{(2r-1)\gamma}}{(\lambda l)^r (\lambda l)^{t+1}}
+ C \tilde C \lambda^m \frac{\lambda^{2r\gamma}}{(\lambda l)^{r+1}} 
\leq C \frac{\lambda^{m}}{(\lambda l)^{r+1}} \tilde C.
\end{split}
\end{equation*}
This ends the proof of \ref{Fbound3K} and thus of all the inductive
estimates, under (\ref{mu}).

\bigskip

{\bf 5. (End of proof)} 
Define $\tilde v\in\mathcal{C}^\infty(\bar\omega +\bar B_l(0),\R^k)$ and
$\tilde w\in\mathcal{C}^\infty(\bar\omega +\bar B_l(0), \R^2)$ according to the ``step''
construction in Lemma \ref{lem_step}, involving the periodic functions
$\Gamma, \bar\Gamma$:
\begin{equation*}
\begin{split}
& \tilde v = v_0 + 
\frac{a_N}{\lambda} \Big(\Gamma(\lambda x_1)e_1 
+ \bar\Gamma(\lambda x_1)e_2 + \Gamma(\lambda x_2)e_3 + \bar\Gamma(\lambda x_2)e_4\Big),\\
& \tilde w = w_{0} - \frac{a_N}{\lambda} 
\Big(\Gamma(\lambda x_1)\nabla v^1_0 
+ \bar\Gamma(\lambda x_1)\nabla v_0^2 + \Gamma(\lambda x_2)\nabla^2 v^3_0 
+ \bar\Gamma(\lambda x_2)\nabla v_0^4 \Big) + \Psi_N - \tilde C id_2.
\end{split}
\end{equation*}
We now show that \ref{Fbound1K} - \ref{Fbound3K} imply the type of bounds
claimed in the Theorem. Observe first that, in virtue of (\ref{mu}):
\begin{equation*}
\begin{split}
& \|\Psi_N\|_1\leq C\|\tilde C\mbox{Id}_2 + \mathcal{D}_0 -
\mathcal{E}_{N-1}\|_{0,\gamma} \leq C\tilde C,
\\ &   \|\nabla^2\Psi_N\|_0\leq C\|\tilde C\mbox{Id}_2 + \mathcal{D}_0 -
\mathcal{E}_{N-1}\|_{1,\gamma} \leq C\Big(\tilde C+ \frac{\tilde C}{l} + 
\frac{\lambda \lambda^\gamma}{\lambda l}\tilde C\Big) \leq C\lambda \tilde C.
\end{split}
\end{equation*}
To prove \ref{Abound1}, we use the above, \ref{Fbound1K}, \ref{Fbound2K} 
and \ref{pr_stima1}, \ref{pr_stima3}:
\begin{equation*}
\begin{split}
& \|\tilde v - v\|_1\leq \|v_0-v\|_1 + C\Big(\|a_N\|_0+\frac{\|\nabla a_N\|_0}{\lambda}\Big)\leq
C \Big(lM +\tilde C^{1/2} + \frac{\lambda^\gamma}{\lambda l}\tilde C^{1/2}\big) \leq C \tilde C^{1/2} ,\\
& \|\tilde w - w\|_1\leq \|w_0-w\|_1 + 
C\Big(\tilde C + \|a_N\|_0\|\nabla v_0\|_0 + \frac{\|\nabla a_N\|_0\|\nabla
  v_0\|_0 + \|a_N\|_0\|\nabla^2 v_0\|_0}{\lambda} \Big) + \|\Psi_N\|_1\\
& \qquad \qquad \leq C \Big(lM+ \tilde C^{1/2} \|\nabla v_0\|_0+
\frac{lM}{\lambda l}\tilde C^{1/2} + \frac{\lambda^\gamma}{\lambda
  l}\tilde C^{1/2} \|\nabla v_0\|_0+ \tilde C \Big)\\
& \qquad \qquad \leq C \tilde C^{1/2}\big( 1+\|\nabla v_0\|_0 + lM+ \tilde C^{1/2}\big)
\leq C \tilde C^{1/2}\big( 1+ \tilde C^{1/2}+\|\nabla v\|_0\big).
\end{split}
\end{equation*}
Similarly, there follows \ref{Abound2}:
\begin{equation*}
\begin{split}
& \|\nabla^2\tilde v\|_0\leq \|\nabla^2 v_0\|_0 +
C\Big(\lambda\|a_N\|_0 + \|\nabla a_N\|_0 +
\frac{\|\nabla^2a_N\|_0}{\lambda}\Big) \\ & \qquad\quad \; \leq
C \Big( M+ \big(\frac{\lambda^\gamma}{l} + \lambda \big)\tilde
C^{1/2}\Big) \leq C\lambda \tilde C^{1/2},\\
&\|\nabla^2\tilde w\|_0\leq \|\nabla^2 w_0\|_0 +
C\Big(\lambda\|a_N\|_0\|\nabla v_0\|_0 + \big(\|\nabla a_N\|_0\|\nabla
  v_0\|_0 + \|a_N\|_0\|\nabla^2 v_0\|_0\big)
\\ & \qquad\qquad\qquad\qquad\qquad \; + \frac{\|\nabla^2a_N\|_0\|\nabla v_0\|_0 + \|\nabla
  a_N\|_0\|\nabla^2 v_0\|_0 + \|a_N\|_0\|\nabla^3 v_0\|_0  }{\lambda}\Big) + \|\nabla^2\Psi_N\|_0 \\ & 
\qquad\quad \; \, \leq C \Big( M+ 
\big(1+\lambda +\frac{1}{l}\big) \tilde C^{1/2}(\|\nabla v\|_0 + lM) + 
\tilde C^{1/2}M + \lambda \tilde C\Big)  
\\ & \qquad\quad \; \, \leq C\lambda \tilde C^{1/2} \big( 1+ \tilde C^{1/2}+\|\nabla v\|_0\big).
\end{split}
\end{equation*}
Finally, (\ref{step_err}) and (\ref{low_bd_asK}) yield \ref{Abound3},
in virtue of the decomposition:
$$\tilde{\mathcal{D}} = (A- A_0) + \mathcal{D}_0 - \Big(\frac{(a_N)^2}{2}
\mathrm{Id}_2 + \mathcal{E}_N +\sym\nabla \Psi_N- \tilde C\Id_2\Big) 
= (A- A_0) - (\mathcal{E}_N - \mathcal{E}_{N-1}),$$
and further, in view of \ref{pr_stima2}, \ref{Fbound3K}:
\begin{equation*}
\|\tilde{\mathcal{D}}\|_0  \leq  \|A-A_0\|_0 + \|\mathcal{E}_N -\mathcal{E}_{N-1}\|_0 \leq
C\Big(l^\beta\|A\|_{0,\beta} + C\frac{\lambda^{2(N-1)\gamma}}{(\lambda l)^N}\tilde C \Big).
\end{equation*}

\medskip

\noindent We now summarize the obtained bounds, under the assumption
$\frac{\lambda l}{\lambda^{2\gamma}} \geq C$, in the following form:
\begin{align*}
& \begin{array}{l}
\|\tilde v - v\|_1\leq C\lambda^{\gamma/2}\big(\|\mathcal{D}\|_0^{1/2} + lM\big), \vspace{1.5mm}\\
\|\tilde w -w\|_1\leq C\lambda^{\gamma}\big(\|\mathcal{D}\|_0^{1/2}
+ lM\big) \big(1+ \|\mathcal{D}\|_0^{1/2} +lM +\|\nabla v\|_0\big), 
\end{array}\vspace{3mm} \\
&  \begin{array}{l}
\|\nabla^2\tilde v\|_0\leq C \lambda
\lambda^{\gamma/2}\big(\|\mathcal{D}\|_0^{1/2} + lM\big), \vspace{1.5mm}\\
\|\nabla^2\tilde w\|_0\leq C \lambda
\lambda^{\gamma}\big(\|\mathcal{D}\|_0^{1/2} + lM\big)  \big(1+\|\mathcal{D}\|_0^{1/2} + lM
+ \|\nabla v\|_0\big), 
\end{array} \vspace{3mm} \\ 
& \begin{array}{l}
\|\tilde{\mathcal{D}}\|_0\leq C\big(l^\beta \|A\|_{0,\beta} +
\displaystyle{\frac{\lambda^{2N \gamma}}{(\lambda l)^N}}\big(
\|\mathcal{D}\|_0 +(lM)^2\big). 
\end{array} 
\end{align*}
The claimed \ref{Abound1} - \ref{Abound3} are obtained by rescaling
$2N\gamma$ to $\gamma$. The proof is done.
\endproof

\section{A proof of Theorem \ref{thm_stageCHI}}\label{sec_stage2}

In the proof below, the constants $C>1$ may change from line to line,
but they depend only on  $\omega$, $\gamma$, $k$,
$m$, $s$, $t$ and $r$ (and thus, ultimately, on $N$), unless specified otherwise.

\bigskip

\noindent {\bf Proof of Theorem \ref{thm_stageCHI}}

\smallskip

{\bf 1. (Preparing the data)} Fix $R, l_0>0$ so that $\bar\omega\subset
B_R$ and $2l_0<1<\mbox{dist} (\bar\omega, \partial B_R)$. Given $l, v, w, A$
as in the statement of the Theorem, we define the smoothed data on
$\bar\omega+\bar B_{2l-l/(2k)}$: 
$$v_0=v\ast \phi_{l/(2k)},\quad w_0=w\ast \phi_{l/(2k)}, \quad A_0=A\ast \phi_{l/(2k)},
\quad {\mathcal{D}}_0= A_0 - \big(\frac{1}{2}(\nabla v_0)^T\nabla v_0 + \sym\nabla w_0\big),$$
where we used the mollifier $\phi_{l/(2k)}$ as in Lemma \ref{lem_stima}.
Similarly to Step 1 of the proof of Theorem \ref{thm_stage}, there follow the
bounds \ref{pr_stima1} - \ref{pr_stima4} with constants $C$ depending
only on $\omega$, $k$ and $m$.

\medskip

\noindent We also set a cut-off function $\chi_0\in
\mathcal{C}^\infty_c(\omega+B_{2l-l/(2k)}, [0,1])$ with $\chi_0= 1$ on
$\bar\omega+\bar B_{2l-l/k}$, and for all $i=1\ldots (k-1)$ we set
the intermediate cut-off functions
$\chi_i\in \mathcal{C}^\infty_c(\omega+B_{2l-il/k}, [0,1])$ with $\chi_i= 1$ on
$\bar\omega+\bar B_{2l-(i+1)l/k}$.
When $l_0\ll 1$, it is pos\-sible to request that for any
$f\in\mathcal{C}^m(\bar\omega+\bar B_{2l-il/k})$ and any multiindex $I$
with $|I|\leq m$, there holds:
\begin{equation}\label{chi_st}
\|\partial_I(\chi_i f)\|_0\leq C \sum_{I_1+ I_2=I}
\|\partial_{I_1}\chi_i\|_0 \|\partial_{I_2}f\|_0\leq C\sum_{I_1+ I_2=I}
\frac{1}{l^{|I_1|}}\|\partial_{I_2}f\|_0,
\end{equation}
with constants $C$ depending only on $\omega$, $k$ and $m$.

\smallskip

\noindent In the course of the proof below, we will inductively construct the intermediate data:
\begin{equation*}
\begin{split}
& v_i \in\mathcal{C}^\infty(\bar\omega + \bar B_{2l-il/k}, \R^k), \quad 
w_i\in\mathcal{C}^\infty(\bar\omega + \bar B_{2l-il/k}, \R^2), \\
& \mathcal{D}_i = A_0 - \big(\frac{1}{2}(\nabla v_i)^T\nabla v_i +
\sym\nabla w_i\big)\in\mathcal{C}^\infty(\bar\omega + \bar B_{2l-il/k},
\R^{2\times 2}_{\sym}) \qquad \mbox{for all }\; i=1\ldots k,
\end{split}
\end{equation*}
and, eventually, set $(\tilde v, \tilde w) =
(v_k, w_k)$. All definitions will rely on the family of pairs of
frequencies $\{\lambda_i, \bar\lambda_i\}_{i=0}^k$ that will be
specified later, assumed to satisfy the monotonicity property:
\begin{equation}\label{freq}
\begin{split}
\lambda_0 = \bar\lambda_0 = \frac{1}{l},\quad \lambda=\lambda_1 \leq
\bar \lambda_1\ldots\leq \lambda_i\leq\bar\lambda_i \ldots
\leq\lambda_k\leq\bar\lambda_k,
\end{split}
\end{equation}
and on the positive constants (where we set $\tilde C_{-1} = 0$):
\begin{equation}\label{Cdef}
\tilde C_{i}  = \frac{2}{r_0}\|\chi_{i}
\mathcal{D}_{i}\|_{0,\gamma} + \bar \lambda_i^\gamma\frac{\big(
  \|\mathcal{D}\|_0+ (lM)^2\big)}{(\bar\lambda_i l)^2}+ 
\tilde C_{i-1}
\bar \lambda_i^\gamma\Big(\frac{\lambda_i^{(N-1)\gamma}}{(\lambda_i/\bar\lambda_{i-1})^{N}}
+ \frac{\lambda_i}{\bar\lambda_i} \Big) 
\quad\mbox{for }\; i=0\ldots k.
\end{equation}
The first term above is crucial, while the other two terms follow from
the technical considerations and are of at most the same order.
We will show that for all $i=1\ldots k$ there holds:
\begin{align*}
& \|v_i - v\|_1\leq C\tilde C_0^{1/2}, \qquad\qquad\quad
\|w_i -w\|_1\leq C\tilde C_0^{1/2} \big(1+\tilde C_0^{1/2} +\|\nabla v\|_0\big), 
\tag*{(\theequation)$_1$}\refstepcounter{equation} \label{INDU1}\medskip\\
& \|\nabla^2 v_i\|_0\leq C\sum_{j=0}^{i-1}\tilde C_j^{1/2}\bar\lambda_{j+1}, \quad \, \,
\|\nabla^2 w_i\|_0\leq C\sum_{j=0}^{i-1}\tilde
C_j^{1/2}\bar\lambda_{j+1}\big(1+\tilde C_0^{1/2} +\|\nabla v\|_0\big),
\tag*{(\theequation)$_2$}\label{INDU2} \medskip\\
&\|\partial_1^{(t)}\partial_2^{(s)} \mathcal{D}_i\|_0 \leq
C \frac{\tilde C_i}{\bar\lambda_i^\gamma}\lambda_i^{t}\bar\lambda_i^{s}
\qquad \mbox{ for all }\; s,t\geq 0. \tag*{(\theequation)$_3$}\label{INDU3}
\end{align*}
Observe that (\ref{chi_st}) and \ref{INDU3} result in:
\begin{equation}\label{pr_stima44}
\|\partial_1^{(t)}\partial_2^{(s)}  (\chi_i\mathcal{D}_i)\|_{0,\gamma}\leq
{C}\tilde C_i \lambda_i^{t}\bar\lambda_i^{s}
\qquad \mbox{ for all }\; s,t\geq 0.
\end{equation}
The bounds \ref{INDU3}, (\ref{pr_stima44})
are already satisfied at $i=0$ because of \ref{pr_stima1} -
\ref{pr_stima4}, and we have:
\begin{equation}\label{pr_now}
\tilde C_0\leq C\frac{1}{l^\gamma}\big(\|\mathcal{D}\|_0 + (lM)^2\big).
\end{equation}

\bigskip

{\bf 2. (Induction definition: iterative decomposition of deficits)} 
We now define the main quantities in the construction of $(v_i, w_i)$ from
$(v_{i-1},w_{i-1})$ for $i=1\ldots k$. 

\smallskip

\noindent Let the linear maps ${\bar a, \bar\Psi}$ be as in 
Lemma \ref{lem_diagonal2}. Also, let $r_0>0$ be a constant depending on
$\omega$, $\gamma$, such that $\|\bar a(D)-1\|_0\leq \frac{1}{2}$
whenever $\|D-\mathrm{Id}_2\|_{0,\gamma}\leq r_0$. For all $i=1\ldots k$
and $r=0\ldots N$, we define the perturbation amplitudes $a^i_r \in\mathcal{C}^\infty(\bar B_R)$ and
the corrections $\Psi^i_r \in\mathcal{C}^\infty(\bar B_R,\R^2)$ by:
\begin{equation*}
\begin{split}
a^{i+1}_0=0, \qquad & (a_r^{i+1})^2= \bar a\Big(\tilde C_i\Id_2 +
\chi_{i} \mathcal{D}_{i} - \chi_{i} \mathcal{E}^{i+1}_{r-1}\Big),
\\ & \Psi^{i+1}_r = \bar\Psi\Big(\tilde C_{i}\Id_2
+\chi_{i} \mathcal{D}_{i}- \chi_{i} \mathcal{E}^{i+1}_{r-1} \Big),
\quad \mbox{ for }\; i=0\ldots k-1, \; r=1\ldots N.
\end{split}
\end{equation*}
The error fields $\mathcal{E}^i_r \in\mathcal{C}^\infty(\bar\omega +
\bar B_{2l-l/(2k)},\R^{2\times 2}_\sym)$ above are given
by the right hand side of (\ref{step_err2}) (taken with $i,j=1$) 
after removing their $[\cdot]_{22}$ entries, namely:
\begin{equation*}
\begin{split}
\mathcal{E}^i_{r}= & - \frac{a_r^i}{\lambda_i} 
\Gamma(\lambda_i x_1)\big(\nabla^2 v_0^i - (\partial_{22}v_0^i )e_2^{\otimes 2}\big)
+\frac{1}{2\lambda_i^2}\Gamma(\lambda_i x_1)^2 \big(\nabla
a_r^i\otimes\nabla a_r^i - (\partial_2 a_r^i)^2 e_2^{\otimes 2}\big) \\ & -
\frac{1}{\lambda_i}\bar \Gamma(\lambda_i x_1) \,\sym\big(\nabla
(a_r^i)^2\otimes e_1\big) \quad\mbox{for }\; i=1\ldots k.
\end{split}
\end{equation*}
Note that $\tilde C_{i}\Id_2 + \chi_{i} \mathcal{D}_{i}- \chi_{i}
\mathcal{E}^{i+1}_{r-1}$ belongs to the space $E$ as in Lemma
\ref{lem_diagonal2}, and so our definition of $a_r^{i+1}$ is correctly posed if only 
$\|\chi_{i}\mathcal{D}_{i} - \chi_{i} \mathcal{E}_{r-1}^{i+1}\|_{0,\gamma} \leq r_0\tilde C_{i}$.
To this end, we right away observe that $\|\chi_{i} \mathcal{D}_{i}\|_{0,\gamma} \leq
r_0\tilde C_{i}/2$, while we will show that the second condition in (\ref{Assu}) implies:
\begin{equation}\label{need_err}
\|\chi_{i}\mathcal{E}_r^{i+1}\|_{0,\gamma}\leq
\frac{r_0\tilde C_{i}}{2} \quad \mbox{ for all } \; 0=1\ldots k-1, \; r=0\ldots N-1.
\end{equation}
Also, then automatically there holds, for all $i=0\ldots k-1$, $r=1\ldots N$:
\begin{equation}\label{low_bd_as}
\begin{split}
& \tilde C_{i}\Id_2 + \chi_{i} \mathcal{D}_{i}- \chi_{i} \mathcal{E}^{i+1}_{r-1}
=(a_r^{i+1})^2\,\mathrm{Id}_2 +\sym\nabla\Psi_r^{i+1}
\\ & \mbox{and }\quad (a_r^{i+1})^2\in \Big[\frac{\tilde C_{i}}{2}, \frac{3\tilde C_{i}}{2}\Big] \;\mbox{ in } \,\bar B_R.
\end{split}
\end{equation}

\smallskip

{\bf 3. (Inductive estimates)} Fix $i=0\ldots k-1$ and assume \ref{INDU3}. 
We right away note that at $i=0$, we indeed have \ref{INDU3} because of \ref{pr_stima4}.
In steps 4-6 below we will prove the
following estimates, valid for all $r=1\ldots N$,
with constants $C$ depending, in line with our convention,
only on $\omega$, $k$, $\gamma$, $t$, $s$, $r$:
\begin{align*}
& \; \, \|\partial_2^{(s)}(a_r^{i+1})^2\|_0\leq {C}\tilde C_i \bar \lambda_i^s\quad\mbox{and}\quad
\|\partial_2^{(s)}a_r^{i+1}\|_0\leq {C}\tilde C_i^{1/2} \bar \lambda_i^s
\qquad \qquad \mbox{ for all }\; s\geq 0,
\tag*{(\theequation)$_1$}\refstepcounter{equation} \label{Fbound1}\medskip\\
& {\begin{array}{l} \|\partial_1^{(t+1)}\partial_2^{(s)} (a_r^{i+1})^2\|_0\leq C
\tilde C_i \displaystyle{\frac{\lambda_{i+1}^{t+1}\bar
  \lambda_i^{s}}{(\lambda_{i+1}/\bar\lambda_i)} }\vspace{2mm} \\ 
\qquad\qquad\qquad \quad \mbox{ and}\quad
\|\partial_1^{(t+1)}\partial_2^{(s)} a_r^{i+1}\|_0\leq C
\tilde C_i^{1/2}\displaystyle{\frac{\lambda_{i+1}^{t+1}\bar
  \lambda_i^{s}}{(\lambda_{i+1}/\bar\lambda_i)} }\end{array}}
\qquad \mbox{ for all }\; s,t\geq 0,
\tag*{(\theequation)$_2$}\label{Fbound2} \medskip\\
&\;\, \|\partial_1^{(t)}\partial_2^{(s)} \big(\mathcal{E}_r^{i+1}-\mathcal{E}_{r-1}^{i+1}\big)\|_0 \leq
C \tilde C_i \lambda_{i+1}^{t}\bar \lambda_i^{s} 
\frac{\lambda_{i+1}^{(r-1)\gamma}}{(\lambda_{i+1}/\bar\lambda_i)^r}
\qquad\qquad\qquad \;\,
\mbox{ for all }\; s,t\geq 0. \tag*{(\theequation)$_3$}\label{Fbound3}
\end{align*}

\noindent We observe that, at any counter value $r\geq 1$, the latter estimates in both \ref{Fbound1} and
\ref{Fbound2} follow from the former ones, in view of 
(\ref{low_bd_as}). Indeed, this implication in \ref{Fbound1} is trivial at $s=0$, whereas for $s\geq
1$ we use the Fa\'a di Bruno formula:
\begin{equation*}
\begin{split}
\|\partial_2^{(s)}a_r^{i+1}\|_0 & \leq C\Big\|\sum_{p_1+2p_2+\ldots
  sp_s=s}(a_r^{i+1})^{2(1/2-p_1-\ldots -p_s)}\prod_{z=1}^s\big|\partial_2^{(z)}(a_r^{i+1})^2\big|^{p_z}\Big\|_0
\\ & \leq C\|a_r^{i+1}\|_0 \sum_{p_1+2p_2+\ldots
  sp_s=s}\prod_{z=1}^s\Big(\frac{\|\partial_2^{(z)}(a_r^{i+1})^2\|_0}{\tilde
  C_i}\Big)^{p_z}\leq C\tilde C_i^{1/2}\bar\lambda_i^s.
\end{split}
\end{equation*}
For \ref{Fbound2}, we apply the multivariate version of the Fa\'a di
Bruno formula. We note that the above estimate is just a particular case
of the more general formula below, but we first separated the more familiar
one-dimensional version for clarity.
Let $\Pi$ be the set of all partitions $\pi$ of the initial
multiindex $\{1\}^{t+1}+\{2\}^s$ into multiindices $I$ of lengths
$|I|\in [0, t+s+1]$ (some of them possibly empty). Denoting
by $|\pi|$ the number of multiindices in the given partition $\pi$, we have:
\begin{equation*}
\begin{split}
\|\partial_1^{(t+1)}\partial_2^{(s)}a_r^{i+1}\|_0 & \leq C\Big\|\sum_{\pi\in
  \Pi} (a_r^{i+1})^{2(1/2-|\pi|)}\prod_{I\in\pi}\partial_I(a_r^{i+1})^2\Big\|_0
 \leq C\|a_r^{i+1}\|_0 \sum_{\pi\in\Pi}\prod_{I\in\pi}\frac{\|\partial_I(a_r^{i+1})^2\|_0}{\tilde C_i}
\\ & \leq C \tilde C_i^{1/2}\lambda_{i+1}^{t+1}\bar\lambda_i^s\sum_{\pi\in\Pi} 
\Big(\prod_{I\in\pi,\;1\in I}\frac{\bar\lambda_i}{\lambda_{i+1}}\Big)
\leq C\tilde C_i^{1/2}\lambda_{i+1}^{t+1}\bar\lambda_i^s \frac{\bar\lambda_i}{\lambda_{i+1}}.
\end{split}
\end{equation*}
Additionally, applying Fa\'a di Bruno's formula to the
inverse rather than square root, we get:
\begin{equation}\label{asm2}
\begin{split}
&\|\partial_2^{(s)}\Big(\frac{1}{a_r^{i+1}+a_{r+1}^{i+1}}\Big)\|_0  \leq
 \frac{C}{\tilde C_i^{1/2}}\sum_{p_1+2p_2+\ldots
  sp_s=s}\prod_{z=1}^s\Big(\frac{\|\partial_2^{(z)}(a_r^{i+1}+a_{r+1}^{i+1})\|_0}{\tilde
  C_i^{1/2}}\Big)^{p_z}\leq \frac{C}{\tilde C_i^{1/2}}\bar\lambda_i^s, \\
&\|\partial_1^{(t+1)}\partial_2^{(s)}\Big(\frac{1}{a_r^{i+1}+a_{r+1}^{i+1}}\Big)\|_0
 \leq C \Big\|\sum_{\pi\in
  \Pi} \big(a_r^{i+1}+a_{r+1}^{i+1})^{-1-|\pi|}\prod_{I\in\pi}\partial_I(a_r^{i+1}+a_{r+1}^{i+1})\Big\|_0
\\ &  \qquad\qquad \qquad\qquad\qquad \qquad 
\leq \frac{C}{\tilde C_i^{1/2}}
 \sum_{\pi\in\Pi}\prod_{I\in\pi}\frac{\|\partial_I a_r^{i+1}\|_0 +\|\partial_Ia_{r+1}^{i+1}\|_0}{\tilde C_i^{1/2}}
\\ & \qquad\qquad \qquad\qquad\qquad \qquad 
\leq \frac{C}{\tilde C_i^{1/2}}\lambda_{i+1}^{t+1}\bar\lambda_i^s 
\sum_{\pi\in\Pi} \Big(\prod_{I\in\pi,\;1\in I}\frac{\bar\lambda_i}{\lambda_{i+1}} \Big) 
\leq \frac{C}{C_i^{1/2}}\lambda_{i+1}^{t+1}\bar\lambda_i^s \frac{\bar\lambda_i}{\lambda_{i+1}}.
\end{split}
\end{equation}

\bigskip

{\bf 4. (Induction base: $r=1$)} In this step, we check that \ref{Fbound1} -- \ref{Fbound3} are 
valid at the lowest counter value $r=1$. Indeed, for all $m\geq 0$
there holds, by Lemma \ref{lem_diagonal2} (iii) and (iv):
\begin{equation}\label{tamr}
\begin{split}
\|\nabla^{(m)} (a_1^{i+1})^2\|_0 & = \sum_{|I|=m} \big\|\bar
a\Big(\partial_I\big(\tilde C_i\mathrm{Id}_2 +\chi_i\mathcal{D}_i\big) 
\Big)\big\|_0\\ & \leq C \|\nabla^{(m)}\big(\tilde C_i\mathrm{Id}_2 
+ \chi_i \mathcal{D}_i \big)\|_{0,\gamma}
\leq C\Big(\tilde C_i + \|\nabla^{(m)}(\chi_i\mathcal{D}_i)\|_{0,\gamma}\Big)
\leq C\tilde C_i\bar\lambda_i^m,
\end{split}
\end{equation}
in view of (\ref{pr_stima44}) and since $\mathcal{E}_0^{i+1}=0$. This
yields, as explained in step 3, that
$\|\nabla^{(m)}a_i^{i+1}\|_0\leq\tilde C_i^{1/2}\bar\lambda_i^m$,
implying \ref{Fbound1} and \ref{Fbound2}, as  $\lambda_{i+1}\geq \bar\lambda_i$.
Further:
\begin{equation*}
\begin{split}
& \|\partial_1^{(t)}\partial_2^{(s)}\mathcal{E}_1^{i+1}\|_0\leq
C\sum_{\tiny\begin{array}{c} p_1+q_1+z_1=t\\ q_2+z_2=s\end{array}}
\lambda_{i+1}^{p_1-1}\|\nabla^{(q_1+q_2)}a_1^{i+1}\|_0\|\nabla^{(z_1+z_2+2)}v_0\|_0 
\\ & \quad + C \hspace{-4mm}\sum_{\tiny\begin{array}{c} p_1+q_1+z_1=t\\ q_2+z_2=s\end{array}}
\lambda_{i+1}^{p_1-2}\|\nabla^{(q_1+q_2+1)} a_1^{i+1}\|_0\|\nabla^{(z_1+z_2+1)}a_1^{i+1}\|_0 + 
C\sum_{p+q=t}\lambda_{i+1}^{p-1}\|\nabla^{(q+s+1)} (a_1^{i+1})^2\|_0
\\ & \leq C\sum_{\tiny\begin{array}{c} p_1+q_1+z_1=t\\  q_2+z_2=s\end{array}}
\Big(\frac{\lambda_{i+1}^{p_1-1}\bar\lambda_i^{q_1+q_2}}{l^{z_1+z_2+1}}\tilde C_i^{1/2}(lM)
+ \lambda_{i+1}^{p_1-2}\bar\lambda_i^{q_1+q_2+z_1+z_2+2}\tilde C_i\Big) +
C\sum_{p+q=t} \lambda_{i+1}^{p-1}\bar\lambda_i^{q+s+1}\tilde C_i
\\ & \leq  C \tilde C_i \lambda_{i+1}^t\bar\lambda_i^s
\sum_{p_1+q_1+z_1=t}\Big(\big(\frac{\bar\lambda_i}{\lambda_{i+1}}\big)^{q_1+z_1+1}
\frac{lM}{\tilde C_i(\bar\lambda_i l)^{z_1+z_2+1}}+
\big(\frac{\bar\lambda_i}{\lambda_{i+1}}\big)^{q_1+z_1+2}\Big)
\\ & \qquad + C\tilde C_i \lambda_{i+1}^t\bar\lambda_i^s
\sum_{p+q=t} \big(\frac{\bar\lambda_i}{\lambda_{i+1}}\big)^{q+1}
\leq C \tilde C_i \frac{\lambda_{i+1}^t\bar\lambda_i^s}{(\lambda_{i+1}/\bar\lambda_i)},
\end{split}
\end{equation*}
for all $s,t\geq 0$, because
$(lM)/(\bar\lambda_i l)\leq \tilde C_i^{1/2}$ by (\ref{Cdef}).
The above is precisely \ref{Fbound3} since $\mathcal{E}_0=0$.

\bigskip

{\bf 5. (Proof of the inductive estimates (\ref{need_err}), \ref{Fbound1} and \ref{Fbound2})} 
Assume that the bounds \ref{Fbound1} -- \ref{Fbound3} hold up to some 
$1\leq r\leq N-1$. We will prove their validity at
$r+1$. We start by noting a direct consequence of \ref{Fbound3} in
view of (\ref{chi_st}) and the interpolation inequality:
\begin{equation}\label{imtih1}
\|\partial_1^{(t)}\partial_2^{(s)}\big(\chi_i (\mathcal{E}_r^{i+1}- \mathcal{E}^{i+1}_{r-1})\big)\|_{0,\gamma}
\leq C\tilde C_i \lambda_{i+1}^t\bar\lambda_i^s
\frac{\lambda_{i+1}^{r\gamma}}{(\lambda_{i+1}/\bar\lambda_i)^r}
\qquad\mbox{ for all }\; s,t\geq 0.
\end{equation}
Further, since $\mathcal{E}_0=0$, it follows that:
\begin{equation}\label{imtih2}
\|\partial_1^{(t)}\partial_2^{(s)}(\chi_i \mathcal{E}_r^{i+1})\|_{0,\gamma}
\leq C\tilde C_i \lambda_{i+1}^t\bar\lambda_i^s
\sum_{j=1}^r \Big(\frac{\lambda_{i+1}^{\gamma}}{\lambda_{i+1}/\bar\lambda_i}\Big)^j
\leq C\tilde C_i \lambda_{i+1}^t\bar\lambda_i^s
\frac{\lambda_{i+1}^{\gamma}}{(\lambda_{i+1}/\bar\lambda_i)}.
\end{equation}
In particular, the requirement (\ref{need_err}) is automatically justified, because:
$$\|\chi_i\mathcal{E}_r^{i+1}\|_{0,\gamma}\leq C\tilde
C_i \frac{\lambda_{i+1}^{\gamma}}{(\lambda_{i+1}/\bar\lambda_i)}
\leq \frac{r_0 \tilde C_i}{2},$$
provided that the following working assumption holds:
\begin{equation}\label{mu2}
\frac{(\lambda_{i+1}/\bar\lambda_i)}{\lambda_{i+1}^{\gamma}}\geq
\frac{2C}{r_0} \quad\mbox{ for all }\; i=0\ldots k-1. 
\end{equation}

\smallskip

\noindent To prove \ref{Fbound1}, we use Lemma \ref{lem_diagonal2}
(iii), (iv) and argue as in (\ref{tamr}) in view of (\ref{imtih2}): 
\begin{equation*}
\begin{split}
& \|\partial_2^{(s)}(a_{r+1}^{i+1})^2\|_0= \big\|\bar
a\Big(\partial_2^{(s)}\big(\tilde C_i\mathrm{Id}_2
+\chi_i\mathcal{D}_i-\chi_i\mathcal{E}_r ^{i+1} \big)\Big)\big\|_0 \\ & 
\leq C\Big(\tilde C_i + \|\partial_2^{(s)}(\chi_i\mathcal{D}_i)\|_{0,\gamma}
+ \|\partial_2^{(s)}(\chi_i\mathcal{E}_r ^{i+1})\|_{0,\gamma}\Big)
\leq C\tilde C_i\Big(1+ \bar\lambda_i^s+ 
\frac{\lambda_{i+1}^{\gamma}\bar\lambda_i^s}{(\lambda_{i+1}/\bar\lambda_i)}\Big)
\leq C\tilde C_i \bar\lambda_i^s.
\end{split}
\end{equation*}
Finally, \ref{Fbound2} follows by additionally invoking Lemma
\ref{lem_diagonal2} (v) in:
\begin{equation*}
\begin{split}
\|\partial_1^{(t+1)}\partial_2^{(s)}(a_{r+1}^{i+1})^2\|_0 & \leq \big\|\bar
a\Big(\partial_1^{(t+1)}\partial_2^{(s)}\big(\tilde C_i\mathrm{Id}_2
+\chi_i\mathcal{D}_i\big)\Big)\big\|_0
+ \big\|\partial_1^{(t+1)}\partial_2^{(s)} \Big(\bar
a\big(\chi_i\mathcal{E}_r ^{i+1} \big)\Big)\big\|_0\\ & 
\leq C\Big(\|\partial_1^{(t+1)}\partial_2^{(s)}\big(\chi_i\mathcal{D}_i\big)\|_{0,\gamma}
+ \|\partial_1^{(t)}\partial_2^{(s+1)}\big(\chi_i\mathcal{E}_r ^{i+1}\big)\|_{0,\gamma}\Big)
\\ & \leq C\tilde C_i\Big(\lambda_{i}^{t+1}\bar\lambda_i^s +
\lambda_{i+1}^t\bar\lambda_i^{s+1}
\frac{\lambda_{i+1}^{\gamma}}{(\lambda_{i+1}/\bar\lambda_i)}\Big)
\leq C\tilde C_i \lambda_{i+1}^t\bar\lambda_i^{s+1}.
\end{split}
\end{equation*}


{\bf 6. (Proof of the inductive estimate \ref{Fbound3})} 
Noting that:
$$(a_{r+1}^{i+1})^2-(a_r^{i+1})^2 = 
-\bar a\big(\chi_i (\mathcal{E}_r^{i+1}-\mathcal{E}^{i+1}_{r-1})\big)$$
and recalling the inductive assumption \ref{Fbound3}, we get from Lemma
\ref{lem_diagonal2} (v) and (\ref{imtih1}):
\begin{equation}\label{uahid}
\begin{split}
&\|\partial_{2}^{(s)} \big((a_{r+1}^{i+1})^2-(a_r ^{i+1})^2)\|_0\leq
C\|\partial_{2}^{(s)} \big(\chi_i (\mathcal{E}_r^{i+1}-\mathcal{E}_{r-1}^{i+1}) \big)\|_{0,\gamma}
\leq C\tilde C_i \bar\lambda_i^{s}
\frac{\lambda_{i+1}^{r\gamma}}{(\lambda_{i+1}/\bar\lambda_i)^r},\\
&\|\partial_1^{(t+1)}\partial_{2}^{(s)} \big((a_{r+1}^{i+1})^2-(a_r^{i+1})^2)\|_0\leq
C\|\partial_1^{(t)}\partial_{2}^{(s+1)} \big(\chi_i (\mathcal{E}_r^{i+1}-\mathcal{E}_{r-1}^{i+1}) \big)\|_{0,\gamma}
\\ & \qquad \qquad\qquad\qquad\qquad \; \;
\leq C\tilde C_i \lambda_{i+1}^{t+1}\bar\lambda_i^{s}
\frac{\lambda_{i+1}^{r\gamma}}{(\lambda_{i+1}/\bar\lambda_i)^{r+1}},
\end{split}
\end{equation}
valid for all $s,t\geq 0$. We now combine the above with (\ref{asm2}) to get:
\begin{equation}\label{ithnain}
\begin{split}
&\|\partial_{2}^{(s)} \big(a_{r+1}^{i+1}-a_r ^{i+1})\|_0 
\leq C \sum_{p+q=s}\|\partial_{2}^{(p)} \big((a_{r+1}^{i+1})^2-(a_r^{i+1})^2)\|_0
\big\|\partial_{2}^{(q)}\Big(\frac{1}{a_{r+1}^{i+1}+a_r^{i+1}}\Big)\big\|_0 
\\ & \qquad \qquad\qquad \qquad\;\;\;
\leq C\tilde C_i^{1/2}\bar\lambda_i^{s}
\frac{\lambda_{i+1}^{r\gamma}}{(\lambda_{i+1}/\bar\lambda_i)^r}, 
\\ & \|\partial_1^{(t+1)}\partial_{2}^{(s)} \big(a_{r+1}^{i+1}-a_r^{i+1})\|_0 
\\ & \qquad \leq C \sum_{\tiny\begin{array}{c} p_1+q_1=t+1\\p_2+q_2=s\end{array}}
\|\partial_1^{(p_1)}\partial_{2}^{(p_2)} \big((a_{r+1}^{i+1})^2-(a_r^{i+1})^2)\|_0
\big\|\partial_1^{(q_1)}\partial_{2}^{(q_2)}\Big(\frac{1}{a_{r+1}^{i+1}+a_r^{i+1}}\Big)\big\|_0 
\\ & \qquad
\leq C\tilde C_i^{1/2} \lambda_{i+1}^{t+1}\bar\lambda_i^{s}
\frac{\lambda_{i+1}^{r\gamma}}{(\lambda_{i+1}/\bar\lambda_i)^{r+1}}
+ C\tilde C_i^{1/2}\hspace{-3mm}\sum_{\tiny\begin{array}{c}
    p_1+q_1=t+1\\p_2+q_2=s, q_1\geq 1\end{array}}
\lambda_{i+1}^{p_1+q_1}\bar\lambda_i^{p_2+q_2}
\frac{\lambda_{i+1}^{r\gamma}}{(\lambda_{i+1}/\bar\lambda_i)^{r+1}}
\\ & \qquad \leq C\tilde C_i^{1/2}\lambda_{i+1}^{t+1}\bar\lambda_i^{s}
\frac{\lambda_{i+1}^{r\gamma}}{(\lambda_{i+1}/\bar\lambda_i)^{r+1}}.
\end{split}
\end{equation}
Towards proving \ref{Fbound3}, we first write:
\begin{equation*}
\begin{split}
\mathcal{E}_{r+1}^{i+1}-\mathcal{E}_r^{i+1} = & -\frac{a_{r+1}^{i+1}-a_{r}^{i+1}}{\lambda_{i+1}}
\Gamma(\lambda_{i+1} x_1)\Big(\nabla^2 v_0^i
-(\partial_{22}v_0^i)e_2^{\otimes 2} \Big)
\\ & +  \frac{1}{2\lambda_{i+1}^2} \Gamma(\lambda_{i+1} x_1)^2
\Big(\big(\nabla a_{r+1}^{i+1} - \nabla a_r^{i+1}\big)\otimes \nabla a_{r+1}^{i+1} + 
\nabla a_r^{i+1} \otimes \big(\nabla a_{r+1}^{i+1} - \nabla a_r^{i+1}\big)
\\ & \qquad\qquad \qquad \qquad \qquad
- \big((\partial_2 a_{r+1}^{i+1})^2-(\partial_2 a_{r}^{i+1})^2\big)e_2^{\otimes 2}\Big)
\\ & - \frac{1}{\lambda_{i+1}}\bar\Gamma(\lambda_{i+1}
x_1)\,\sym\big(\nabla((a_{r+1}^{i+1})^2-(a_r^{i+1})^2)\otimes e_1\big),
\end{split}
\end{equation*}
with the goal of estimating, for all $s,t\geq 0$:
\begin{equation}\label{thalatha}
\begin{split}
\|\partial_1^{(t)}&\partial_2^{(s)}\big(\mathcal{E}_{r+1}^{i+1}-\mathcal{E}_r^{i+1}\big)\|_0 
\\  & \leq  C \sum_{\tiny\begin{array}{c} p_1+q_1+z_1=t\\q_2+z_2=s\end{array}}
\lambda_{i+1}^{p_1-1}\|\partial_1^{(q_1)}\partial_2^{(q_2)}(a_{r+1}^{i+1}-a_r^{i+1})\|_0
\big\|\nabla^{(z_1+z_2+2)}v_0\big\|_0 
\\ & \quad + C \sum_{\tiny\begin{array}{c} p_1+q_1+z_1=t\\q_2+z_2=s\end{array}}
\lambda_{i+1}^{p_1-2}\|\partial_1^{(q_1)}\partial_2^{(q_2)}\nabla(a_{r+1}^{i+1}-a_r^{i+1})\|_0 \times
\\ & \qquad\qquad\qquad \qquad\qquad
\times \Big(\|\partial_1^{(z_1)}\partial_2^{(z_2)}\nabla a_{r}^{i+1}\|_0+
\|\partial_1^{(z_1)}\partial_2^{(z_2)}\nabla a_{r+1}^{i+1}\|_0\Big)
\\ & \quad + C \sum_{p+q=t}
\lambda_{i+1}^{p-1}\|\partial_1^{(q)}\partial_2^{(s)}\nabla ((a_{r+1}^{i+1})^2-(a_r^{i+1})^2)\|_0. 
\end{split}
\end{equation}
The first term above is bounded, in virtue of (\ref{ithnain}) and \ref{pr_stima2}, by:
\begin{equation*}
\begin{split}
& C\tilde C_i^{1/2} (lM)\lambda_{i+1}^t\Big(\hspace{-2mm}
\sum_{\tiny\begin{array}{c} p_1+z_1=t\\q_2+z_2=s\end{array}}
\frac{\bar\lambda_i^{q_2}}{\lambda_{i+1}^{z_1+1}}\frac{\lambda_{i+1}^{r\gamma}}{(\lambda_{i+1}
/ \bar\lambda_i)^{r}}\frac{1}{l^{z_1+z_2+1}}
\\ & \qquad\qquad\qquad \qquad\qquad + \sum_{\tiny\begin{array}{c} p_1+q_1+z_1=t\\q_1\geq 1\\q_2+z_2=s\end{array}}
\frac{\bar\lambda_i^{q_2}}{\lambda_{i+1}^{q_1+z_1+1}}
\frac{\lambda_{i+1}^{q_1}\lambda_{i+1}^{r\gamma}}{(\lambda_{i+1} / \bar\lambda_i)^{r+1}}\frac{1}{l^{z_1+z_2+1}}\Big)
\\ & \leq C\tilde C_1^{1/2} (lM)\lambda_{i+1}^t\bar\lambda_i^s
\Big(\sum_{p_1+z_1=t}
\frac{1}{(\lambda_{i+1} l)^{z_1+1}}\frac{\lambda_{i+1}^{r\gamma}}{(\lambda_{i+1}/\bar\lambda_i)^{r}}
+ \sum_{p_1+q_1+z_1=t} \frac{1}{(\lambda_{i+1} l)^{z_1+1}}\frac{\lambda_{i+!}^{r\gamma}}{(\lambda_{i+1}
 /\bar\lambda_i)^{r+1}}\Big)
\\ & \leq C\tilde C_i \lambda_{i+1}^t\bar\lambda_i^s \frac{\lambda_{i+1}^{r\gamma}}{(\lambda_{i+1}
 /\bar\lambda_i)^{r+1}},
\end{split}
\end{equation*}
in virtue of (\ref{Cdef}). Similarly, the third term in (\ref{thalatha}) is estimated through (\ref{uahid}), by:
\begin{equation*}
C\tilde C_i \lambda_{i+1}^t\Big(\frac{\bar\lambda_{i}^{s+1}}{\lambda_{i+2}}
\frac{\lambda_{i+1}^{r\gamma}}{(\lambda_{i+1}/ \bar\lambda_i)^{r}}
+ \sum_{p+q=t,\;q\geq 1} 
\bar\lambda_i^s\frac{\lambda_{i+1}^{r\gamma}}{(\lambda_{i+1}/ \bar\lambda_i)^{r+1}}
\leq C\tilde C_i \lambda_{i+1}^t\bar\lambda_i^s
\frac{\lambda_{i+1}^{r\gamma}}{(\lambda_{i+1}/ \bar\lambda_i)^{r+i}}.
\end{equation*}
For the middle term in (\ref{thalatha}), we obtain the bound:
\begin{equation*}
\begin{split} 
& C\tilde C_i \lambda_{i+1}^t \Bigg(\sum_{q_2+z_2=s}\frac{1}{\lambda_{i+1}^{2}}
\frac{\bar\lambda_i^{q_2+1}\lambda_{i+1}^{r\gamma}}{(\lambda_{i+1}/ \bar\lambda_i)^{r}}\bar\lambda_i^{z_2+1}
+\sum_{\tiny\begin{array}{c} p_1+z_1=t\\z_1\geq 1\\ q_2+z_2=s\end{array}}
\frac{1}{\lambda_{i+1}^{z_1+2}}\frac{\bar\lambda_i^{q_2+1}\lambda_{i+1}^{r\gamma}}{(\lambda_{i+1}/
  \bar\lambda_i)^{r}}\lambda_{i+1}^{z_1}\bar\lambda_i^{z_2+1}
\\ &\qquad \qquad \quad + \sum_{\tiny\begin{array}{c} p_1+q_1=t\\q_1\geq 1\\q_2+z_2=s\end{array}}
\frac{1}{\lambda_{i+1}^{q_1+2}}\frac{\lambda_{i+1}^{q_1+1}\bar\lambda_i^{q_2}\lambda_{i+1}^{r\gamma}}{(\lambda_{i+1}/
  \bar\lambda_i)^{r+1}}\bar\lambda_i^{z+2+1} \\ &
\qquad \qquad \quad + \sum_{\tiny\begin{array}{c}
    p_1+q_1+z_1=t\\q_1\geq 1, z_1\geq 1\\q_2+z_2=s\end{array}}
\frac{1}{\lambda_{i+1}^{q_1+z_1+2}}
\frac{\lambda_{i+1}^{q_1+1}\bar\lambda_i^{q_2}\lambda_{i+1}^{r\gamma}}{(\lambda_{i+1}/\bar\lambda_i)^{r+1}}
\lambda_{i+1}^{z_1}\bar\lambda_i^{z_2+1}\Bigg)
\\ & \leq C\tilde C_i
\lambda_{i+1}^t\bar\lambda_i^{s}\frac{\lambda_{i+1}^{r\gamma}}{(\lambda_{i+1}/
  \bar\lambda_i)^{r+2}},
\end{split}
\end{equation*}
where we repeatedly used (\ref{ithnain}) and \ref{Fbound1}, \ref{Fbound2}.
In conclusion, (\ref{thalatha}) yields \ref{Fbound3K}:
$$\|\partial_1^{(t)}\partial_2^{(s)}\big(\mathcal{E}_{r+1}^{i+1}-\mathcal{E}_r^{i+1}\big)\|_0 
\leq C\tilde C_i
\lambda_{i+1}^t\bar\lambda_i^{s}\frac{\lambda_{i+1}^{r\gamma}}{(\lambda_{i+1}/
  \bar\lambda_i)^{r+1}}.$$
This ends the proof of all the inductive estimates, under the assumption (\ref{mu}).

\bigskip

{\bf 7. (Adding the first corrugation)} 
Define the perturbed fields $\bar v_i\in\mathcal{C}^\infty(\bar\omega +\bar B_{2l- il/k},\R^k)$ and
$\bar w_i\in\mathcal{C}^\infty(\bar\omega +\bar B_{2l- il/k}, \R^2)$
in accordance with Lemma \ref{lem_step2} (here, $i=1\ldots k$):
\begin{equation*}
\begin{split}
& \bar v_i = v_{i-1} + \frac{a_N^i}{\lambda_i} \Gamma(\lambda_i x_1)e_i,\quad 
\bar w_{i} = w_{i-1} - \frac{a_N^i}{\lambda_i} \Gamma(\lambda_i x_1)\nabla v^i_0 
+\frac{(a_N^i)^2}{\lambda_i}\bar\Gamma(\lambda_i x_1)e_1
+ \Psi_N^i - \tilde C_{i-1} id_2.
\end{split}
\end{equation*}
Observe that, recalling \ref{Fbound1}, \ref{Fbound2}, there holds for all $s,t\geq 0$:
\begin{equation}\label{jamea0}
\begin{split}
\|\partial_1^{(t)}\partial_2^{(s)}(\bar v_i^i - v_{i-1}^i)\|_0& \leq C\sum_{p+q=t}\lambda_i^{p-1}
\|\partial_1^{(q)}\partial_2^{(s)}a_N^i\|_0\\ & \leq C\tilde C_{i-1}^{1/2}\lambda_i^t\bar\lambda_{i-1}^s
\Big(\frac{1}{\lambda_i} + \sum_{p+q=t, ~q\geq
  1}\frac{\bar \lambda_{i-1}}{\lambda_i^2}\Big) 
\leq C\tilde C_i \lambda_i^t\bar\lambda_{i-1}^s\frac{1}{\lambda_i}.
\end{split}
\end{equation}
In particular, by \ref{pr_stima1}, \ref{pr_stima3}, the above yields (since $v_{i-1}^i=v_0^i$):
\begin{equation}\label{jamea1}
\begin{split}
& \|\bar v_i^i - v^i\|_1\leq \|v_0-v\|_1 + \|\bar v_i^i-v_{i-1}^i\|_1\leq
C(lM + C_{i-1}^{1/2}) \leq C \tilde C_0^{1/2} ,\\
& \|\partial_1^{(t)}\partial_2^{(s)}\nabla^2\bar v_i^i\|_0 
\leq \|\nabla^{(t+s+2)}v_0\|_0 + C\tilde C_i^{1/2}\lambda_i^{t+1}\bar\lambda_{i-1}^s
\leq C\tilde C_i^{1/2}\lambda_i^{t+1}\bar\lambda_{i-1}^s,
\end{split}
\end{equation}
by (\ref{pr_now}) so that $(lM)^2\leq C\tilde C_0$, and by
(\ref{Cdef}) which yields $lM\leq \tilde C_{i-1}^{1/2}(\lambda_il)$,
and as long as:
\begin{equation}\label{urid2}
\tilde C_{i-1}\leq C\tilde C_0.
\end{equation}
Now, by Lemma \ref{lem_diagonal2},
(\ref{pr_stima44}), (\ref{imtih2}) and (\ref{mu2}) we obtain:
\begin{equation*}
\begin{split}
& \|\Psi_N^i\|_1\leq C\big(\|\chi_{i-1}\mathcal{D}_{i-1} \|_{0,\gamma}+
\|\chi_{i-1}\mathcal{E}^i_{N-1}\|_{0,\gamma}\big) \leq C\tilde C_{i-1},
\\ &   \|\nabla^2\Psi_N^i\|_0\leq C\big(\|\chi_{i-1}\mathcal{D}_{i-1} \|_{1,\gamma}+
\|\chi_{i-1}\mathcal{E}^i_{N-1}\|_{1,\gamma}\big) \leq C\Big(\tilde C_{i-1}\bar\lambda_{i-1} + 
\frac{\lambda_i \lambda^{\gamma}}{(\lambda_i/\bar\lambda_i)}\tilde C\Big) \leq C \tilde C_{i-1}\lambda_i,
\end{split}
\end{equation*}
and consequently, \ref{Fbound1}, \ref{Fbound2}  \ref{pr_stima1},
\ref{pr_stima3} result in the estimate:
\begin{equation}\label{jamea2}
\begin{split}
& \|\bar w_i - w_{i-1}\|_1\leq C\Big(\tilde C_{i-1} + \|a_N^i\|_0\|\nabla v_0\|_0 + \|(a_N^i)^2\|_0 
\\ & \qquad\qquad \qquad\qquad \qquad \; \, +\frac{\|\nabla a_N^i\|_0\|\nabla
  v_0\|_0 + \|a_N^i\|_0\|\nabla^2 v_0\|_0 + \|\nabla (a_N^i)^2\|_0}{\lambda_i} \Big) \\
& \qquad \qquad \quad \; \;  \leq C \Big(\tilde C_{i-1}+ \tilde C_{i-1}^{1/2}
\|\nabla v_0\|_0 + \frac{\tilde
  C_{i-1}^{1/2}}{\bar\lambda_{i-1}/ \lambda_i}\|\nabla v_0\|_0 +
\frac{lM}{\lambda_i l}\tilde C_{i-1}^{1/2} + \tilde C \Big)\\
& \qquad \qquad \quad \; \; \leq C \tilde C_{i-1}^{1/2}\big( \|\nabla v_0\|_0 + \tilde C_{i-1}^{1/2}\big)
\leq C \tilde C_{i-1}^{1/2}\big( lM+ \tilde C_{i-1}^{1/2}+\|\nabla v\|_0\big)
\\ & \qquad \qquad \quad \; \; \leq C \tilde C_0^{1/2}\big(\tilde C_0^{1/2}+\|\nabla v\|_0\big),\\
& \|\nabla^2 \bar  w_i- \nabla^2 \bar  w_{i-1}\|_0\leq \|\nabla^2\Psi_N\|_0 
+ C\Big(\lambda_i\big(\|a_N^i\|_0\|\nabla v_0\|_0+ \|(a_N^i)^2\|_0\big) 
\\ & \qquad\qquad\qquad 
+ \big(\|\nabla a_N^i\|_0\|\nabla v_0\|_0 + \|a_N^i\|_0\|\nabla^2 v_0\|_0 + \|\nabla (a_N^i)^2\|_0\big)
\\ & \qquad\qquad\qquad + \frac{\|\nabla^2a_N^i\|_0\|\nabla v_0\|_0 + \|\nabla
  a_N^i\|_0\|\nabla^2 v_0\|_0 + \|a_N^i\|_0\|\nabla^3 v_0\|_0
  +\|\nabla^2(a_N^i)^2\|_0 }{\lambda_i}\Big) \\ &  
\qquad\quad \; \; \; \leq C \Big(\lambda_i \tilde C_{i-1} +
\big(\lambda_i +\bar\lambda_{i-1}\big) \tilde C_{i-1}^{1/2}(\|\nabla v\|_0 + lM) + 
\tilde C_{i-1}^{1/2}\frac{lM}{l} \Big)  
\\ & \qquad\quad \; \; \;  \leq C\lambda_i \tilde C_{i-1}^{1/2} \big(lM+ \tilde C_{i-1}^{1/2}+\|\nabla v\|_0\big)
\leq C\tilde C_{i-1}\lambda_i\big(\tilde C_0^{1/2}+\|\nabla v\|_0\big).
\end{split}
\end{equation}
Finally, we note that (\ref{step_err2}) and (\ref{low_bd_as}) yield the
decomposition on $\bar\omega + \bar B_{2l-il/k}$:
\begin{equation}\label{jamea3}
\begin{split}
\bar {\mathcal{D}}_i & \doteq A_0 - \big(\frac{1}{2}(\nabla \bar  v_i)^T\nabla \bar  v_i +
\sym\nabla \bar  w_i\big) \\ &  = \mathcal{D}_{i-1} - \Big((a_N^i)^2
e_1\otimes e_1+ \mathcal{E}_N^i  +\sym\nabla \Psi_N^i- \tilde C_{i-1}\Id_2\Big)
\\ & \quad - \Big(-\frac{a^i_N}{\lambda_i} \Gamma(\lambda_i x_1)\partial_{22}v_0^i
+ \frac{1}{2\lambda_i^2}\Gamma(\lambda_i x_1)^2 (\partial_2a_N^i)^2\Big) e_2\otimes e_2
\\ & = - \big(\mathcal{E}_N^i- \mathcal{E}_{N-1}^i  \big)
+ \Big((a_N^i)^2 +\frac{a^i_N}{\lambda_i} \Gamma(\lambda_i x_1)\partial_{22}v_0^i
- \frac{1}{2\lambda_i^2}\Gamma(\lambda_i x_1)^2
(\partial_2a^i_N)^2\Big) e_2{\otimes} e_2.
\end{split}
\end{equation}

\medskip

{\bf 8. (Adding the second corrugation)} We now update of
$\bar v_i, \bar w_i$, to new fields $v_i\in\mathcal{C}^\infty(\bar\omega +\bar B_{2l-il/k},\R^k)$,
$ w_i\in\mathcal{C}^\infty(\bar\omega +\bar B_{2l-il/k}, \R^2)$
by using Lemma \ref{lem_step2} (with $i=2, j=1$) and the perturbation 
amplitude dictated by (\ref{jamea3}), namely:
\begin{equation}\label{vw_fin}
\begin{split}
& v_i = \bar v_i + \frac{b^i(x)}{\bar\lambda_i} \Gamma(\bar\lambda_i x_2)e_i,\quad 
w_1 = \bar w_{1} - \frac{b^i(x)}{\bar\lambda_i} \Gamma(\bar\lambda_i x_2)\nabla \bar  v^i_i 
+\frac{b^i(x)^2}{\bar\lambda_i}\bar\Gamma(\bar\lambda_i x_2)e_2,
\\ & \mbox{where }\; 
(b^i)^2 = (a_N^i)^2 +\frac{a_N^i}{\lambda_i} \Gamma(\lambda_i x_1)\partial_{22}v_0^i
- \frac{1}{2\lambda_i^2}\Gamma(\lambda_i x_1)^2 (\partial_2a_N^i)^2.
\end{split}
\end{equation}
Firstly, we argue that $b\in\mathcal{C}^\infty(\bar\omega + \bar B_{2l-il/k})$ is well
defined, since by \ref{pr_stima3}, \ref{Fbound1}, \ref{Fbound2}:
$$\big\|\frac{a_N^i}{\lambda_i} \Gamma(\lambda x_1)\partial_{22}v_0^i
- \frac{1}{2\lambda_i^2}\Gamma(\lambda_i x_1)^2 (\partial_2a_N^i)^2 \big\|_0\leq 
C\Big(\tilde C_{i-1}^{1/2}\frac{(lM)}{\lambda_i l} + \tilde C_{i-1}\Big(\frac{\bar\lambda_{i-1}}{\lambda_i}\Big)^2
 \Big)\leq C\tilde C_{i-1} \frac{\bar\lambda_{i-1}}{\lambda_i}.$$
Assigning $\bar\lambda_{i-1}/\lambda_i$ small as guaranteed by (\ref{mu2}), the bound in
(\ref{low_bd_as}) implies:
\begin{equation}\label{b0}
(b^i)^2\in \Big[\frac{\tilde C_{i-1}}{3}, 2\tilde C_{i-1}\Big] \quad \mbox{ in } \; \bar\omega + \bar B_{2l-il/k}.
\end{equation}
Likewise, for every $t,s\geq 0$ we obtain the first set of bounds in:
\begin{equation}\label{b1}
\begin{split}
& \|\partial_2^{(s)}(b^i)^2\|_0 \leq C\tilde C_{i-1}\bar\lambda_{i-1}^s
\qquad \|\partial_1^{(t+1)}\partial_2^{(s)}(b^i)^2\|_0 
\leq C\tilde C_{i-1}\lambda_i^t\bar\lambda_{i-1}^{s+1}, \\
& \|\partial_2^{(s)}b^i\|_0 \leq C\tilde C_{i-1}^{1/2}\bar\lambda_{i-1}^s,
\qquad \;\; \; \|\partial_1^{(t+1)}\partial_2^{(s)}b^i\|_0 
\leq C\tilde C_{i-1}^{1/2}\lambda_i^t\bar\lambda_{i-1}^{s+1},
\end{split}
\end{equation}
while the second set of bounds follows by the Fa\'a di Bruno formula as in
step 4. Further:
\begin{equation}\label{jamea01}
\begin{split}
& \|\partial_2^{(s)}(v_i^i - \bar v_i^i)\|_0 \leq C\sum_{p+q=s}\bar\lambda_i^{p-1}
\|\partial_2^{(q)}b^i\|_0\leq C\tilde C_{i-1}^{1/2}\frac{\bar\lambda_i^s}{\bar\lambda_i},\\ 
& \|\partial_1^{(t+1)}\partial_2^{(s)}(v_i^i - \bar v_i^i)\|_0 \leq C\sum_{p+q=s}\bar\lambda_i^{p-1}
\|\partial_1^{(t+1)}\partial_2^{(q)}b^i\|_0\\ & \qquad \qquad\qquad 
\qquad\quad \, \leq C\tilde C_{i-1}^{1/2} \bar\lambda_i^s\sum_{p+q=s}
\frac{\lambda_i^{t}\bar\lambda_{i-i}^{q+1}}{\lambda_i^{q+1}}\leq C \tilde C_{i-1}^{1/2}
\lambda_i^t\lambda_i^s\frac{\bar\lambda_{i-1}}{\bar\lambda_i},
\end{split}
\end{equation}
as in (\ref{jamea0}), which combined with (\ref{jamea1}) implies that:
\begin{equation}\label{jamea11}
\begin{split}
& \|v_i^i - v^i\|_1\leq \|\bar v_i^i-v^i\|_1 + C\tilde C^{1/2}_{i-1} \leq C \tilde C_0^{1/2} ,\\
& \|\nabla^2v_i^i\|_0  \leq C\tilde C_{i-1}^{1/2}\Big(\lambda_i + \bar\lambda_i +  
\bar\lambda_{i-1}\Big) \leq C\tilde C_{i-1}^{1/2}\bar\lambda_i,
\end{split}
\end{equation}
and further:
\begin{equation}\label{jamea1.5}
\begin{split}
& \|w_i - w_{i-1}\|_1\leq \|\bar w_i-w_{i-1}\|_1 
\\ & \qquad\qquad 
+ C\Big(\|b^i\|_0\|\nabla \bar v_i^i\|_0 + \|(b^i)^2\|_0 + \frac{\|\nabla b^i\|_0\|\nabla
  \bar v_i^i\|_0 + \|b^i\|_0\|\nabla^2 \bar v_i^i \|_0+ \|\nabla (b^i)^2\|_0}{\bar\lambda_i} \Big) \\
& \leq C \tilde C_{i-1}^{1/2}\big(\tilde C_0^{1/2 }+ \tilde
C_{i-1}^{1/2}+\|\nabla v\|_0\big)\leq C \tilde C_{i-1}^{1/2}\big(\tilde C_0^{1/2 }+\|\nabla v\|_0\big) , \\
& \|\nabla^2 w_i-\nabla^2 w_{i-1} \|_0\leq \|\nabla^2 \bar w_i - \nabla^2 w_{i-1} \|_0 
+ C\Big(\bar\lambda_i\big(\|b^i\|_0\|\nabla \bar v_i^i\|_0+ \|(b^i)^2\|_0\big) 
\\ & \qquad\qquad
+ \|\nabla b^i\|_0\|\nabla \bar v_i^i\|_0 + \|b^i\|_0\|\nabla^2 \bar v_i^i\|_0 + \|\nabla (b^i)^2\|_0
\\ & \qquad\qquad
+ \frac{\|\nabla^2b^i\|_0\|\nabla \bar v_i^i\|_0 + \|\nabla
  b^i\|_0\|\nabla^2 \bar v_i^i\|_0 + \|b^i\|_0\|\nabla^3\bar v_i^i\|_0
  +\|\nabla^2(b^i)^2\|_0 }{\bar\lambda_i}\Big) 
\\ & \leq C \tilde C_{i-1}^{1/2} \bar\lambda_i
\big(\tilde C_0^{1/2} + \tilde C_{i-1}^{1/2} + \|\nabla v\|_0\big)\leq 
C \tilde C_{i-1}^{1/2} \bar\lambda_i \big(\tilde C_0^{1/2} +\|\nabla v\|_0\big),
\end{split}
\end{equation}
where we used $\|\nabla^3\bar v_i^i\|_0\leq C\tilde C_{i-1}^{1/2}\lambda_i^2$.

\smallskip

\noindent Finally, Lemma \ref{lem_step2} implies:
\begin{equation}\label{kalb}
\begin{split}
&\big(\frac{1}{2}(\nabla v_i)^T\nabla v_i +
\sym\nabla w_i\big) - \big(\frac{1}{2}(\nabla \bar v_i)^T\nabla \bar v_i +
\sym\nabla \bar w_i\big) = (b^i)^2e_2\otimes e_2 +\mathcal{F}^i\\ &
\mbox{with }\; \mathcal{F}^i = 
-\frac{b^i}{\bar\lambda_i} \Gamma(\bar\lambda_i x_2)\nabla^2\bar v_i^i
+ \frac{1}{2\bar\lambda_i^2}\Gamma(\bar\lambda_i x_2)^2 \nabla b^i\otimes \nabla b^i -
\frac{1}{\bar\lambda_i} \bar\Gamma(\bar\lambda_i x_2)\,\sym\big(\nabla (b^i)^2\otimes e_2\big),
\end{split}
\end{equation}
where the bounds (\ref{b1}), (\ref{jamea1}) yield for all $s,t\geq 0$:
\begin{equation*}
\begin{split}
\|\partial_2^{(s)}\mathcal{F}^i\|_0& \leq C\sum_{p+q+z=s}\bar\lambda_i^{p-1}\|\partial_2^{(q)}b^i\|_0
\|\partial_2^{(z)}\nabla^2\bar v_i^i\| \\ & \quad +
C\sum_{p+q+z=s}\bar\lambda_i^{p-2}\|\partial_2^{(q)}\nabla b^i\|_0
\|\partial_2^{(z)}\nabla b^i\|_0+  C\sum_{p+q=s}\bar\lambda_i^{p-1}\|\partial_2^{(q)}\nabla (b^i)^2\|_0 \\ &
\leq  C\tilde C_{i-1}\bar\lambda_i^s\Big(
\sum_{p+q+z=s}\frac{\bar\lambda_{i-1}^q\lambda_i\bar\lambda_{i-1}^z}{\bar\lambda_i^{q+z+1}}
+ \sum_{p+q+z=s}\frac{\bar\lambda_{i-1}^{q+1}\bar\lambda_{i-1}^{z+1}}{\bar\lambda_i^{q+z+2}} + 
\sum_{p+q=s}\frac{\bar\lambda_{i-1}^{q+1}}{\bar\lambda_i^{q+1}}
\Big)\leq C \tilde C_{i-1}\bar\lambda_i^s\frac{\lambda_i}{\bar\lambda_i},
\end{split}
\end{equation*}
\begin{equation*}
\begin{split}
& \|\partial_i^{(t+1)}\partial_2^{(s)}\mathcal{F}^i\|_0 \leq 
C\sum_{\tiny\begin{array}{c} p_1+q_1+z_1=s\\  q_2+z_2=t+1\end{array}}\bar\lambda_i^{p_i-1}
\|\partial_i^{(q_2)}\partial_2^{(q_1)}b ^i\|_0
\|\partial_1^{(z_2)}\partial_2^{(z_1)}\nabla^2\bar v_i^i\| 
\\ & \qquad \qquad\qquad \qquad + C\sum_{\tiny\begin{array}{c} p_1+q_1+z_1=s\\  q_2+z_2=t+1\end{array}}
\bar\lambda_i^{p_1-2}\|\partial_1^{(q_2)}\partial_2^{(q_1)}\nabla b^i\|_0 
\|\partial_1^{(z_2)}\partial_2^{(z_1)}\nabla b^i\|_0
\\ & \qquad \qquad\qquad \qquad
+ C\sum_{p+q=s}\bar\lambda_i^{p-1}\|\partial_1^{(t+1)}\partial_2^{(q)}\nabla (b^i)^2\|_0 
\\ & \leq  C\tilde C_{i-1}\bar\lambda_i^s\Big(\sum_{p_1+q_1+z_1=s} 
\frac{\bar\lambda_{i-1}^{q_1}\lambda_i^{t+2}\bar\lambda_{i-1}^{z_1}}{\bar\lambda_{i}^{q_1+z_1+1}}
+ \sum_{\tiny\begin{array}{c} p_1+q_1+z_1=s\\ q_2+z_2=t+1, ~ q_2\geq 1\end{array}}
\frac{\bar\lambda_{i-1}^{q_1+1}\lambda_i^{q_2-1}}{\bar\lambda_i^{q_1+z_1+1}}\lambda_i^{z_2+1}\bar\lambda_{i-1}^{z_1}
\\ & \qquad\qquad \qquad  + \sum_{p_1+q_1+z_1=s}\frac{\bar\lambda_{i-1}^{q_1+1}}{\bar\lambda_i^{q_1+z_1+2}}
\lambda_i^{t+1}\bar\lambda_{i-1}^{z_1+1}
+ \sum_{\tiny\begin{array}{c} p_1+q_1+z_1=s\\ q_2+z_2=t+1, ~ q_2,  z_2\geq 1\end{array}}
\frac{\lambda_i^{q_2}\bar\lambda_{i-1}^{z_2}\lambda_{i-1}^{z_1+1}\bar\lambda_{i-1}^{q_1+1}}
{\bar\lambda_i^{q_1+z_1+2}}
\\ & \qquad\qquad \qquad   
+ \sum_{p+q=s}\frac{\lambda_i^{t+1}\bar\lambda_{i-1}^{q+1}}{\bar\lambda_i^{q+1}} \Big)
\leq C \tilde C_{i-1}\bar\lambda_i^{t+1}\lambda_i^s\frac{\lambda_i}{\bar\lambda_i}.
\end{split}
\end{equation*}
Since the new defect can be written in virtue of (\ref{jamea3}) and (\ref{kalb}) as:
\begin{equation*}
\begin{split}
\mathcal{D}_i & = \bar{\mathcal{D}}_i - \Big(\big(\frac{1}{2}(\nabla v_i)^T\nabla v_i +
\sym\nabla w_i\big) - \big(\frac{1}{2}(\nabla \bar v_i)^T\nabla
\bar v_i + \sym\nabla \bar w_1\big)\Big) 
\\ & = - (\mathcal{E}^i_{N}-\mathcal{E}^i_{N-1}) - \mathcal{F}^i,
\end{split}
\end{equation*}
then \ref{Fbound3} together with the above developed bounds on the derivatives of $\mathcal{F}$ imply:
\begin{equation}\label{jamea4}
\begin{split}
\|\partial_{1}^{(t)}\partial_2^{(s)}\mathcal{D}_i \|_0 & \leq 
C\tilde C_{i-1}\lambda_i^t\bar\lambda_{i-1}^s \frac{ \lambda_i^{(N-1)\gamma}}{(\lambda_i/\bar\lambda_{i-1})^{N}}
+ C\tilde C_{i-1}\lambda_i^t\bar\lambda_{i}^s
\frac{\lambda_i}{\bar\lambda_i} \\ & \leq 
C\tilde C_{i-1}\lambda_i^t\bar\lambda_{i}^s 
\Big(\frac{\lambda_i^{(N-1)\gamma}}{(\lambda_i/\bar\lambda_{i-1})^{N}} + \frac{\lambda_i}{\bar\lambda_i} \Big)
\quad \mbox{ for all } \; t,s\geq 0.
\end{split}
\end{equation}

\medskip

{\bf 9. (Conclusion of the proof)} From (\ref{jamea4}) we see that
\ref{INDU3} is satisfied, together with:
\begin{equation}\label{jamea5}
\tilde C_i\leq C\tilde C_{i-1}
\bar \lambda_i^\gamma\Big(\frac{\lambda_i^{(N-1)\gamma}}{(\lambda_i/\bar\lambda_{i-1})^{N}}
+ \frac{\lambda_i}{\bar\lambda_i} \Big) + \bar
\lambda_i^\gamma\frac{\big(\|\mathcal{D}\|_0+
  (lM)^2\big)}{(\bar\lambda_i l)^2}\quad\mbox{ for all }\; i=1\ldots k.
\end{equation}
Also, from (\ref{jamea11}), (\ref{jamea1.5}) we get \ref{INDU1},
\ref{INDU2}, provided that (\ref{mu2}), (\ref{urid2}) hold, namely:
\begin{equation}\label{Urid}
\tilde C_{i}\leq C\tilde C_0 \quad \mbox{and}\quad
\frac{(\lambda_{i}/\bar\lambda_{i-1})}{\lambda_{i}^{\gamma}}\geq \frac{2C}{r_0}
\quad\mbox{ for all } i=1\ldots k.
\end{equation}
We now make a scaling assumption, towards obtaining the final
form of the constants $\tilde C_i$. Set:
\begin{equation}\label{freq1}
\frac{\bar\lambda_i}{\lambda_i} =
\Big(\frac{\lambda_i}{\bar\lambda_{i-1}}\Big)^N\quad\mbox{ for all } i=1\ldots k,
\end{equation}
which equivalently reads: $\lambda_i^{N+1} = \bar\lambda_i
\bar\lambda_{i-1}^N$ and further: $\big({\lambda_i}/{\bar\lambda_{i-1}}\big)^N =
\big({\bar\lambda_i}/{\bar\lambda_{i-1}}\big)^{N/(N+1)}.$
We also assume the following condition:
\begin{equation}\label{Urid1}
\frac{(\bar\lambda_{i}/\bar\lambda_{i-1})}{\bar\lambda_i^{\gamma N
    (N+1)}}\geq \Big(\frac{2C}{r_0}\Big)^{N+1} 
\quad\mbox{ for all } i=1\ldots k.
\end{equation}
Observe that the above implies the second condition in (\ref{Urid}) because:
$$\frac{\lambda_{i}/\bar\lambda_{i-1}}{\lambda_{i}^{\gamma}}\geq 
\frac{\lambda_{i}/\bar\lambda_{i-1}}{\bar\lambda_{i}^{\gamma}}
= \Big(\frac{\bar\lambda_{i}/\bar\lambda_{i-1}}{\bar\lambda_{i}^{(N+1)\gamma}}\Big)^{1/(N+1)}
\geq \frac{2C}{r_0},$$
whereas the first condition in (\ref{Urid}) also then holds, as
(\ref{jamea5}) becomes in view of (\ref{Urid1}):
\begin{equation*}
\begin{split}
\tilde C_i & \leq C\tilde
C_{i-1}\frac{\bar\lambda_i^{N\gamma}}{(\lambda_i/\bar\lambda_{i-1})^N}
+ \bar\lambda_i^\gamma\frac{\big(\|\mathcal{D}\|_0+
  (lM)^2\big)}{(\bar\lambda_il)^2} \\ & \leq C \frac{\tilde
  C_{i-1}}{(\bar\lambda_i/\bar\lambda_{i-1})^{(N-1)/(N+1)}} + 
\bar\lambda_i^\gamma\frac{\big(\|\mathcal{D}\|_0+ (lM)^2\big)}{(\bar\lambda_il)^2},
\end{split}
\end{equation*}
so that a straightforward induction argument shows that:
\begin{equation}\label{jamea6}
\tilde C_i \leq C \frac{\tilde C_{0}}{(\bar\lambda_i/\bar\lambda_{0})^{(N-1)/(N+1)}}= 
C \frac{\tilde C_{0}}{(\bar\lambda_il)^{(N-1)/(N+1)}}
\quad\mbox{ for all } i=1\ldots k.
\end{equation}
Also, we note that the monotonicity of the sequence:
\begin{equation}\label{freq2}
\frac{1}{l}=\bar\lambda_0 \leq \frac{(\lambda l)^{N+1}}{l} =
\bar\lambda_1\leq \bar\lambda_2\ldots \leq\bar\lambda_i\ldots \leq\bar\lambda_k,
\end{equation}
implies the monotonicity properties in (\ref{freq}), because then  $\lambda_i
= (\bar\lambda_i\bar\lambda_{i-1}^N)^{1/(N+1)}\geq \bar\lambda_{i-1}$ in view of
(\ref{freq1}) and thus also $\bar\lambda_i/\lambda_i \geq 1$, for all $i=1\ldots k$.

\medskip

\noindent Finally, we assign the progression of frequences in
(\ref{freq2}) motivated by \ref{INDU2}, namely we request:
\begin{equation}\label{cosia}
\tilde C_{i-1}^{1/2}\bar\lambda_i \leq C \tilde
C_0^{1/2}\bar\lambda_1 \quad\mbox{ for all } i=1\ldots k,
\end{equation}
which in view of (\ref{jamea6}) is implied by:
$\bar\lambda_il\leq C (\bar\lambda_il) (\bar\lambda_{i-1}l)^{(N-1)/(2(N+1))}$. We thus set:
$$\bar\lambda_il=C (\bar\lambda_il) (\bar\lambda_{i-1}l)^{(N-1)/(2(N+1))}.$$
The above is a straightforward recursion, which in the closed form yields the formula:
$$\bar\lambda_il = (\bar\lambda_1l)^{\alpha_i} \quad\mbox{with }\;
\alpha_i =\frac{1-\Big(\frac{N-1}{2(N+1)}\Big)^i}{1-\frac{N-1}{2(N+1)}}
\quad\mbox{ for all } i=1\ldots k, $$
that indeed is compatible with (\ref{freq2}).
Towards verifying (\ref{Urid1}), the above implies:
$$\frac{\bar\lambda_{i}/\bar\lambda_{i-1}}{\bar\lambda_{i}^{N(N+1)\gamma}}
= \frac{(\bar\lambda_1l)^{\alpha_i-\alpha_{i-1}}}
{(\bar\lambda_1l)^{N(N+1)\gamma\alpha_i}}l^{N(N+1)\gamma}
= \frac{(\lambda l)^{(N+1)(\alpha_i - \alpha_{i-1}-
    N(N+1)\gamma\alpha_i + N\gamma)}}{\lambda^{N(N+1)\gamma}},$$
Since the exponent in the numerator term above, for every $i=1\ldots
k$ can be estimated by:
$$\alpha_i - \alpha_{i-1}- N(N+1)\gamma\alpha_i + N\gamma\geq 
\Big(\frac{N-1}{2(N+1)}\Big)^{k}- \gamma N(2N+3)\geq \frac{1}{2}
\Big(\frac{N-1}{2(N+1)}\Big)^{k},$$
if only $\gamma$ is sufficiently small in function of $N$ and $k$, we
see that (\ref{Urid1}) is implied by:
\begin{equation}\label{Urid2}
\frac{\lambda l}{\lambda^{2N
\big(\frac{2(N+1)}{N-1}\big)^{k}\gamma}}\geq \Big(\frac{2C}{r_0}
\Big)^{2\big(\frac{2(N+1)}{N-1}\big)^k},
\end{equation}
while \ref{INDU1}-\ref{INDU3} result in:
\begin{align*}
& \|v_k - v\|_1\leq C\tilde C_0^{1/k}, \qquad\qquad\quad
\|w_k -w\|_1\leq C\tilde C_0^{1/2}\big(1+\tilde C_0^{1/2} +\|\nabla v\|_0\big), \\
& \|\nabla^2 v_k\|_0\leq C\tilde C_0^{1/2}\bar\lambda_{1}, \qquad\qquad 
\|\nabla^2 w_k\|_0\leq C\tilde C_0^{1/2}\bar\lambda_1
\big(1+\tilde C_0^{1/2} +\|\nabla v\|_0\big),\\
&\|\mathcal{D}_k\|_0 \leq C \frac{\tilde
  C_k}{\bar\lambda_k^\gamma}\leq C \frac{\tilde C_0}{(\lambda l)^{(N-1)\alpha_k}}.
\end{align*}

\noindent We now summarize the obtained bounds, under the assumption
(\ref{Urid2}), in the following form:
\begin{align*}
& \begin{array}{l}
\| v_k- v\|_1\leq C\lambda^{\gamma/2}\big(\|\mathcal{D}\|_0^{1/2} + lM\big), \vspace{1.5mm}\\
\|w_k -w\|_1\leq C\lambda^{\gamma}\big(\|\mathcal{D}\|_0^{1/2}
+ lM\big) \big(1+ \|\mathcal{D}\|_0^{1/2} +lM +\|\nabla v\|_0\big), 
\end{array}\vspace{3mm} \\
&  \begin{array}{l}
\|\nabla^2v_k\|_0\leq C \displaystyle{\frac{(\lambda
  l)^{N+1}\lambda^{\gamma/2}}{l}}\big(\|\mathcal{D}\|_0^{1/2} + lM\big), 
\vspace{1.5mm}\\ 
\|\nabla^2w_k\|_0\leq C \displaystyle{\frac{(\lambda l)^{N+1}\lambda^{\gamma}}{l}}
\big(\|\mathcal{D}\|_0^{1/2} + lM\big)  \big(1+\|\mathcal{D}\|_0^{1/2} + lM+ \|\nabla v\|_0\big), 
\end{array} \vspace{3mm} \\ 
& \begin{array}{l}
\|\mathcal{D}_k\|_0\leq C \displaystyle{\frac{\lambda^{\gamma}}
{(\lambda l)^{\frac{2(N^2-1)}{N+3}\big(1-\big(\frac{N-1}{2(N+1)}\big)^k\big)}}}\big(
\|\mathcal{D}\|_0 +(lM)^2\big)
\end{array} 
\end{align*}
The claimed \ref{Abound12} - \ref{Abound32} follow by rescaling
$2N\big(\frac{N-1}{2(N+1)}\big)^k\gamma$ to $\gamma$. The proof is done.
\endproof

\section{A proof of Theorem \ref{th_final}}\label{sec4}

The proof of Theorem \ref{th_final} 
relies on iterating Theorems \ref{thm_stage} and \ref{thm_stageCHI} according to the
Nash-Kuiper scheme, whose proof and estimates involving the H\"older
exponent, as in the decomposition Lemma \ref{lem_diagonal}, were given
\cite{lew_improved}. We need to further adjust these iteration
estimates in view of the new assumption in (\ref{Assu}). Recall the following: 

\begin{theorem}\label{th_old}\cite[Theorem 1.4]{lew_improved}
Let $\omega\subset\R^d$ be open, bounded and smooth,
and let $k, J, S\geq 1$. Assume that there exists $l_0\in (0,1)$ such that
the following holds for every $l\in (0, l_0]$. Given
$v\in\mathcal{C}^2(\bar\omega+\bar B_{2l}(0), \R^k)$,  $w\in\mathcal{C}^2(\bar\omega+\bar B_{2l}(0), \R^d)$, 
$A\in\mathcal{C}^{0,\beta}(\bar\omega+\bar B_{2l}(0), \R^{d\times
  d}_\sym)$, and $\gamma, \lambda, M$ with:
\begin{equation}\label{ass_impro}
\gamma\in (0,1),\qquad \lambda>\frac{1}{l},\qquad M\geq
\max\{\|v\|_2, \|w\|_2, 1\},
\end{equation}
there exist  $\tilde v\in\mathcal{C}^2(\bar \omega+\bar B_l(0),\R^k)$,
$\tilde w\in\mathcal{C}^2(\bar\omega+\bar B_l(0),\R^d)$ satisfying:
\begin{align*}
& \hspace{-3mm} \left. \begin{array}{l} \|\tilde v - v\|_1\leq
C\lambda^{\gamma/2}\big(\|\mathcal{D}\|_0^{1/2}+lM\big), \vspace{1mm} \\ 
\|\tilde w - w\|_1\leq C\lambda^{\gamma}\big(\|\mathcal{D}\|_0^{1/2}+lM\big)
\big(1+ \|\mathcal{D}\|_0^{1/2}+lM+\|\nabla v\|_0\big), \end{array}\right.
\vspace{5mm}\\
& \hspace{-3mm} \left. \begin{array}{l} \|\nabla^2\tilde v\|_0\leq C{\displaystyle{\frac{(\lambda
 l)^J}{l}\lambda^{\gamma/2}}}\big(\|\mathcal{D}\|_0^{1/2}+lM\big),\vspace{1mm}\\ 
\|\nabla^2\tilde w\|_0\leq C{\displaystyle{\frac{(\lambda
  l)^J}{l}}}\lambda^{\gamma}\big(\|\mathcal{D}\|_0^{1/2}+lM\big)
\big(1+\|\mathcal{D}\|_0^{1/2}+lM+\|\nabla v\|_0\big), \end{array}\right.
\medskip\\ 
& \|\tilde{\mathcal{D}}\|_0\leq C\Big(l^\beta{\|A\|_{0,\beta}}
+\frac{\lambda^{\gamma}}{(\lambda l)^S} \big(
\|\mathcal{D}\|_0 + (lM)^2\big)\Big).
\end{align*}
with constants $C$ depending only on $\omega, k, J,S,\gamma$,
and with the defects, as usual, denoted by: $\mathcal{D}=A -\big(\frac{1}{2}(\nabla v)^T\nabla v + \sym\nabla
w\big)$ and $\tilde{\mathcal{D}}=A -\big(\frac{1}{2}(\nabla \tilde
v)^T\nabla \tilde v + \sym\nabla \tilde w\big)$.

\smallskip

\noindent Then, for every $v, w, A$ as above which additionally
satisfy $0<\|\mathcal{D}\|_0\leq 1$, and for every $\alpha$ in:
\begin{equation}\label{rangeAlz}
0< \alpha <\min\Big\{\frac{\beta}{2},\frac{S}{S+2J}\Big\},
\end{equation}
there exist $\bar v\in\mathcal{C}^{1,\alpha}(\bar\omega,\R^k)$ and
$\bar w\in\mathcal{C}^{1,\alpha}(\bar\omega,\R^d)$ with the following properties:
\begin{align*}
& \|\bar v - v\|_1\leq C \big(1+\|\nabla v\|_0\big)^2
\|\mathcal{D}_0\|_0^{1/4}, \quad \|\bar w -
w\|_1\leq C(1+\|\nabla v\|_0)^3\|\mathcal{D}\|_0^{1/4}, \vspace{1mm}\\
& A-\big(\frac{1}{2}(\nabla \bar v)^T\nabla \bar v + \sym\nabla
\bar w\big) =0 \quad\mbox{ in }\; \bar\omega. 
\end{align*}
The constants $C$ above depend only on $\omega, k, A$ and $\alpha$. 
\end{theorem}

\smallskip

\noindent The above formulation does not allow iterating on the
construction in Theorems \ref{thm_stage} and \ref{thm_stageCHI}, because of the 
assumption $\lambda^{1-\gamma} l\geq \sigma_0$. We however observe:

\begin{lemma}\label{lem_NK}
Theorem \ref{th_old}  remains valid if (\ref{ass_impro}) is replaced
by a more restrictive assumption:
\begin{equation}\label{ass_impro2}
\gamma\in (0,1),\qquad \lambda^{1-\gamma} l>\sigma_0,\qquad M\geq
\max\{\|v\|_2, \|w\|_2, 1\},
\end{equation}
where $\sigma_0\geq 1$ is a given constant depending on $\omega,
k, S, J,\gamma$.
\end{lemma}
\begin{proof}
We only indicate changes in the proof of Theorem
\ref{thm_stage} in \cite{lew_improved}, referring to the formulas
numbering in there. In step 1, the general requirement (4.3) is now replaced by:
\begin{equation*}
\begin{split}
& l_{i+1}\leq \frac{l_{i}}{2},\quad  l_i\lambda_i^{1-\gamma}>\sigma_0,
\quad M_i\geq \max\{\|v_i\|_2,\|w_i\|_2,1\},\quad M_i\nearrow \infty,
\\ & \|\mathcal{D}_i\|_0\leq (l_iM_i)^2, \quad l_iM_i\to 0,
\end{split}
\end{equation*}
taking into account the middle assumption in (\ref{ass_impro2}).
In step 2, formulas (4.6), (4.8) and (4.11) are augmented by the
additional requirement on $b$ in the definition $\lambda_i = b/{l_i^a}$:
$$b^{1/2}\geq \sigma_0.$$
Steps 3 - 5 remain unchanged. In steps 6 and 7, we augment (4.22) and
(4.24) by the same bound above, and proof of their
viability is the same. Step 8 remains unaltered.
\end{proof}

\smallskip

\noindent We note that Lemma \ref{lem_NK} automatically yields the
following result below, where we compute $\frac{S}{S+2J} =
\frac{1}{1+2J/S}$, with: $\frac{J}{S}=\frac{1}{N}\to 0$ as $N\to\infty$
when $k\geq 4$, while in the general case:
\begin{equation*}
\frac{J}{S}=\frac{N+3}{2(N-1)
  \big(1-\big(\frac{N-1}{2(N+1)}\big)^k\big)}\to
\frac{1}{2\big(1-\frac{1}{2^k}\big)}= \frac{2^{k-1}}{2^k-1} \quad \mbox{ as } N\to\infty
\end{equation*}

\begin{corollary}\label{th_NKH}
Let $\omega\subset\R^2$ be an open, bounded and smooth domain, and
let $k\geq 1$. Fix any $\alpha$ as in (\ref{VKrange}).
Then, there exists $l_0\in (0,1)$ such that, for every $l\leq l_0$,
given $v\in\mathcal{C}^2(\bar\omega + \bar B_{2l}(0),\R^k)$,
$w\in\mathcal{C}^2(\bar\omega +\bar B_{2l}(0),\R^2)$,
$A\in\mathcal{C}^{0,\beta}(\bar\omega +\bar B_{2l}(0), \R^{2\times 2}_\sym)$, such that:
$$\mathcal{D}=A-\big(\frac{1}{2}(\nabla v)^T\nabla v + \sym\nabla
w\big) \quad\mbox{ satisfies } \quad 0<\|\mathcal{D}\|_0\leq 1,$$
there exist $\tilde v\in\mathcal{C}^{1,\alpha}(\bar\omega,\R^k)$,
$\tilde w\in\mathcal{C}^{1,\alpha}(\bar\omega,\R^2)$ with the following properties:
\begin{align*}
& \|\tilde v - v\|_1\leq C(1+\|\nabla v\|_0)^2\|\mathcal{D}\|_0^{1/4}, \quad \|\tilde w -
w\|_1\leq C (1+\|\nabla v\|_0)^3\|\mathcal{D}\|_0^{1/4}, \vspace{1mm}\\
& A-\big(\frac{1}{2}(\nabla \tilde v)^T\nabla \tilde v + \sym\nabla
\tilde w\big) =0 \quad\mbox{ in }\; \bar\omega. 
\end{align*}
The norms in
the left hand side above are taken on $\bar\omega$, and in the right hand
side on $\bar\omega+ \bar B_{2l}(0)$. The constants $C$ depend only
on $\omega, k, A$ and $\alpha$. 
\end{corollary}

\bigskip

\noindent The proof of Theorem \ref{th_final} is consequently the same as the proof of
Theorem 1.1 in \cite{lew_improved}, in section 5 in there. We 
replace $\omega$ by its smooth superset, and apply the basic stage
construction in order to first decrease $\|\mathcal{D}\|_0$
below $1$. Then, Corollary \ref{th_NKH} yields the result. \endproof

\end{document}